\documentclass[pdflatex,sn-nature]{sn-jnl}

\usepackage{enumitem}
\usepackage{graphicx}%
\usepackage{multirow}%
\usepackage{amsmath,amssymb,amsfonts}%
\usepackage{amsthm}%
\usepackage{mathrsfs}%
\usepackage[title]{appendix}%
\usepackage{xcolor}%
\usepackage{textcomp}%
\usepackage{manyfoot}%
\usepackage{booktabs}%
\usepackage{algorithm}%
\usepackage{algorithmicx}%
\usepackage{algpseudocode}%
\usepackage{listings}%

\theoremstyle{thmstyleone}%
\newtheorem{theorem}{Theorem}[section]
\newtheorem{lemma}{Lemma}[section]
\newtheorem{corollary}{Corollary}[section]

\theoremstyle{thmstyletwo}%
\newtheorem{example}{Example}%

\theoremstyle{thmstylethree}%
\newtheorem{definition}{Definition}[section]%

\newtheorem{assumption}{Assumption}[section]

\usepackage{hyperref}

\newcommand{\rb}[1]{ \left( #1 \right) }
\newcommand{\rrb}[1]{ \left\lbrace #1 \right\rbrace }
\newcommand{\curbr}[1]{ \lbrace #1 \rbrace }
\newcommand{\cb}[1]{ \left[ #1 \right] }
\newcommand{\abs}[1]{\left|#1\right|}

\newcommand{\norm}[1]{\left\lVert#1\right\rVert}

\newcommand{\suml}{\sum\limits}

\newcommand{\intl}{\int\limits}

\newcommand{\liml}{\lim\limits}
\newcommand{\supl}{\sup\limits}

\newcommand{\LandauO}{ \mathcal{O} }

\newcommand{\bareps}{ {\ddot \mu} }

\newcommand{\sgn}{\operatorname{sgn}}

\newcommand{\IndNr}[1]{\mathds{1}_{\left\lbrace#1\right\rbrace}}

\newcommand{\Si}{\operatorname{Si}}

\newcommand{\C}{\mathbb{C}}

\newcommand{\Kbb}{\mathbb{K}}

\newcommand{\K}{\mathbb{K}}

\newcommand{\N}{\mathbb{N}}

\newcommand{\R}{\mathbb{R}}

\newcommand{\Z}{\mathbb{Z}}

\usepackage{dsfont}
\newcommand{\ONE}{\mathds{1}}

\newcommand{\Mu}{\operatorname{M}}

\begin{document}

\title[Deconvolution through symmetrization]{Deconvolution of arbitrary distribution functions and densities}

\author*[1]{\fnm{Henrik} \sur{Kaiser}}\email{hkaiser@icloud.com}

\affil*[1]{\orgaddress{\city{Schotten}, \postcode{63679}, \state{Hesse}, \country{Germany}}}

\abstract{In this article we propose a novel approach for the deconvolution of the distribution function associated with an arbitrary probability measure (and possibly existing density). We first show that the initial convolution equation always can be transformed to a convolution equation that involves a symmetric distribution function, whose characteristic function has its values in the unit interval. As a consequence, the characteristic function of the target measure turns out as the limit of a geometric series. By truncation of this series, appro\-xi\-ma\-tions for distribution function and density are established. The convergence properties of these approximations are examined in detail across diverse setups.}




\maketitle

\section{Introduction}

Consider the distribution functions (d.f. or d.fs., for short) $F_\chi$ and $F_\mu$ associated with arbitrary probability measures $\chi$ and $\mu$. It is well-known that then also the convolution $\chi\ast\mu$ constitutes a probability measure, whose d.f. is represented by the additive convolution of $F_\chi$ and $F_\mu$, that is
\begin{align} \label{Verteilungsfaltung}
F_{\chi\ast\mu}(\xi) = \intl_{-\infty}^\infty F_\chi(\xi-z)F_\mu(dz) \hspace{1cm} (\xi \in \R).
\end{align}
As there is no danger of confusion, we omit the prefix ''additive'' and simply speak of convolution. Whenever $F_\chi$ possesses a (Lebesgue) density $f_\chi$, the density corresponding to $F_{\chi\ast\mu}$ is
\begin{align} \label{Dichtenfaltung}
f_{\chi\ast\mu}(\xi) = \intl_{-\infty}^\infty f_\chi(\xi-z) F_\mu(dz) \hspace{1cm} (\xi\in\R),
\end{align}
where $F_\mu(dz) = f_\mu(z) dz$, if $F_\mu$ also is absolutely continuous with density $f_\mu$. In older literature, a more common notion for the above integrals is (Stieltjes) resultant or the German word Faltung (see, e.g., p. 51--52 in \cite{titchmarsh1937} or p. 84 in \cite{widder1946}). Alternatively, $F_{\chi\ast\mu}$ and $f_{\chi\ast\mu}$ can be conceived as the convolution transforms  with respect to the d.f. $F_\mu$, respectively of $F_\chi$ and $f_\chi$ (cf. \cite{HirschmanWidder1955}). Following the common notion from applications, we refer to $\mu$ as some kind of error or noise measure. The recovery of $F_\chi$ or $f_\chi$, given $F_{\chi\ast\mu}$ or $f_{\chi\ast\mu}$ and $F_\mu$ is mostly the same. It is a well-known mathematical problem with a long history, in modern literature called deconvolution (cf. Ch. XI in \cite{titchmarsh1937}, $\S$1.9 in \cite{tricomi1985integral} or Ch. V, $\S$8 in \cite{widder1946}). Convolution is particularly associated with Fourier analysis, as it simplifies to a multiplicative product in the Fourier domain. In fact, in terms of characteristic functions (c.f. or c.fs., for brevity), each of the above convolution equations implies that
\begin{align} \label{additFiVcF}
\Phi_{\chi\ast\mu}(t) = \Phi_\chi(t)\Phi_\mu(t) \hspace{1cm} (t \in \R).
\end{align}
Whenever $t \in \R$ with $\Phi_\mu(t) \neq 0$, this identity is equivalent to the quotient
\begin{align} \label{nastyquotient}
\Phi_\chi(t) = \frac{\Phi_{\chi\ast\mu}(t)}{\Phi_\mu(t)}.
\end{align}
It shows that $\Phi_\chi(t)$ is uniquely determined for all $t \in \R$, only if $\Phi_\mu(t) \neq 0$, for Lebesgue almost every $t\in \R$. Then, the above representation holds Lebesgue almost everywhere and elsewhere by continuity. Furthermore, through inversion of $\Phi_\chi$, one readily returns to $F_\chi$. This approach becomes infeasible if there exists $(t_1,t_2)\subset\R$ with $\Phi_\mu(t)=0$ for any $t \in (t_1, t_2)$. In this event, $\Phi_\chi(t)$ is indeterminable, for all $t \in (t_1, t_2)$, and hence $F_\chi$ is eventually unidentifiable. Closely related to this problem is the unboundedness of the inverse operator, when considering convolution on function spaces. For that reason, in operator theory, according to \cite{Hadamard1902}, deconvolution is considered an ill-posed inverse problem.
\\
\hspace*{1em}The aim of this text consists in establishing an alternative to the ordinary Fourier inversion formula for the recovery of $F_\chi$ from its convolution transform $F_{\chi\ast\mu}$. Our focus lies in fact on d.fs., as these exist without loss of ge\-ne\-rality. Starting point of our approach is the observation that the right hand side of the identity (\ref{nastyquotient}) represents the limit of a geometric series, whenever $\abs{1-\Phi_\mu(t)}<1$. Yet, as this is rarely the case, after a short discussion of preliminaries in \S\ref{SecNotNPrel}, we invoke a symmetrizing and thereby convergence generating factor in \S\ref{SecSymDecProb}. The existence of such a factor relies on the principle that the product of an arbitrary c.f. with its complex conjugate again establishes a c.f., that is associated with a symmetric distribution and particularly non-negative. As a result of our symmetrization, we arrive at a transformed convolution equation, which eventually facilitates the desired geometric series expansion in the Fourier domain. By truncation of this expansion, a novel signed d.f. is established, the so-called deconvolution function. Basic properties of the deconvolution function are reviewed in \S\ref{SubSecDekma}, like absolute continuity with respect to the Lebesgue measure in some cases. Representations as Fourier integrals are introduced in \S\ref{SecFourierIntDecFct}, with the aid of which, in \S\ref{ChConvDecFct}, convergence is discussed, as the truncation index grows to infinity. As a consequence thereof, the deconvolution function and its derivative represent appro\-xi\-ma\-tions for the target d.f. $F_\chi$ and for the possibly existing density $f_\chi$. The article is concluded by \S\ref{SubSummary} with a summary and an outlook on future results. Several figures throughout the text illustrate the performance of the newfound functions.

\section{Notation and preliminaries} \label{SecNotNPrel}

Throughout the text, if $Q : \R \rightarrow \C$ is an arbitrary function, we indicate the limit from the left and from the right at $\xi \in \R$ by $Q(\xi-)$ and by $Q(\xi+)$, respectively. The set of discontinuities of $Q$ is denoted by $D_Q$, i.e., $\xi \in D_Q$ if and only if $Q(\xi+)\neq Q(\xi-)$. Conversely, the associated continuity points/intervals are $C_Q := \R \setminus D_Q$. Also, if existent, $Q(\pm\infty) := \lim_{\xi \rightarrow \pm\infty} Q(\xi)$. Particularly if $Q$ is continuous on $\R$ and both of these limits exist, it is continuous on $\overline \R := \R \cup \curbr{\pm \infty}$. Furthermore, $\norm{ Q }_p$, for $0 < p \leq \infty$, refers to the $L^p$-norm and $\Delta(A, B) := \inf_{ (a, b) \in A \times B } \abs{a-b}$ to the distance of two sets $A, B \subseteq\R$. The Dirac measure with mass at $a \in \R$ is represented by $\delta_{\curbr{a}}$, whereas $\ONE_\mathcal{M}$ stands for the indicator of the set $\mathcal{M} \subset \overline \R$. In the usual fashion, we use the big $\LandauO$ and small $o$ notation, and we indicate by $\Re z$, $\Im z$ and $\overline z$, respectively, the real part, the imaginary part and the complex conjugate of $z \in \C$. We moreover write $\abs{Q}([a, b])$ for the variation of $Q$ over the interval $[a, b] \subset \overline \R$ (compare $\S$2.1 in \cite{wheeden2015measure}), with a straightforward extension to infinite intervals, if $Q(\pm\infty)$ exists. In particular, if $Q$ has a continuous derivative $Q'$ on $[a,b]$, equivalently $\abs{Q}([a, b]) = \int_a^b \abs{ Q'(t) } dt$. In any case, $Q$ is said to be of bounded variation on $[a, b]$, if $\abs{Q}([a, b]) < \infty$.
\\
\hspace*{1em}Denote by $\mathcal{M}(\K, \mathcal{B}(\R))$ the vector space of signed (if $\Kbb = \R$) or complex (if $\Kbb = \C$) measures, with $\mathcal{B}(\R)$ being the Borel $\sigma$-algebra on $\R$. We write $D_\nu := \curbr{ x \in \R : \nu(\curbr{x})\neq0}$ for the point masses of $\nu \in \mathcal{M}(\K, \mathcal{B}(\R))$, also called atoms. Each $\nu \in \mathcal{M}(\K, \mathcal{B}(\R))$ induces a d.f., written $F_\nu(\xi):=\nu((-\infty, \xi])$, for $\xi \in \R$. We specifically define $F_\nu\curbr{\xi} := F_\nu(\xi+) - F_\nu(\xi-)$, so that $\xi$ is a discontinuity of $F_\nu$ if and only if $F_\nu\curbr{\xi}\neq0$. All discontinuities of $F_\nu$ are finite and exactly coincide with the atoms of $\nu$. Basically, integration with respect to $F_\nu$ and $\nu$ is the same, i.e., $F_\nu(dx)=\nu(dx)$. Moreover, we refer to $|\nu|(E)$ as the total variation on $E \in \mathcal{B}(\R)$ of $\nu$ and note that $|\nu|(\R) < \infty$, by construction of this vector space (see, e.g., $\S$9A in \cite{axler2019measure}). This notion is in fact equi\-va\-lent to the variation of functions. On the one hand, if $Q$ is an arbitrary function of bounded variation on $[a, b]\subseteq\overline\R$, then $\nu_Q(E) := \int_{E \cap [a, b]} Q(dx)$, for $E \in \mathcal{B}(\R)$, constitutes a (signed or complex) measure, i.e., $\nu_Q \in \mathcal{M}(\K, \mathcal{B}(\R))$. On the other hand, the (signed or complex) d.f. $F_\nu$ of any $\nu \in \mathcal{M}(\K, \mathcal{B}(\R))$ is of bounded variation on $\overline \R$. In the sequel, absolute continuity always refers to the Lebesgue measure. If $\nu \in \mathcal{M}(\K, \mathcal{B}(\R))$ is in fact absolutely continuous, we indicate the associated density by $f_\nu \in L^1(\R)$, i.e., $\nu(dx) = f_\nu(x)dx$. Of particular interest in this text are probability measures, which are precisely $\nu \in \mathcal{M}(\K, \mathcal{B}(\R))$ with $\nu\geq0$ and $\nu(\R) = 1$. For these, $\Mu_\nu(k) := \int_{-\infty}^\infty x^k \nu(dx)$, with $k \in \N_0$, stands for the $k$-th moment. Moreover, the convolution of $\nu_1, \nu_2 \in \mathcal{M}(\K, \mathcal{B}(\R))$ refers to the integral $(\nu_1 \ast \nu_2)(E) := \int_\R \int_\R \ONE_E(x+y) \nu_1(dx) \nu_2(dy)$, for $E \in \mathcal{B}(\R)$, and is well-defined for any Borel set. If $E := (-\infty,\xi]$, for $\xi\in\R$, we arrive at the convolution of the associated d.fs. $(F_{\nu_1} \ast F_{\nu_2})(\xi) = \int_\R F_{\nu_1}(\xi-x) F_{\nu_2}(dx)$, with $F_{\nu_1} \ast F_{\nu_2} = F_{\nu_1 \ast \nu_2}$. Convolution of signed measures obviously constitutes some kind of commutative product, the neutral element being the Dirac measure with mass at the origin $\delta_{\curbr{0}}$. Accordingly, it makes sense to define $\nu^{\ast 0} := \delta_{ \curbr{0} }$ and by $\nu^{\ast k} := \nu \ast \nu^{\ast(k-1)}$, for $k \in \N$, the $k$-th convolution power of $\nu$. In this fashion, one can also verify the binomial convolution identity
\begin{align} \label{BinConvId}
(\delta_{\curbr{0}} - \nu)^{\ast \ell} = \suml_{k=0}^\ell \binom{\ell}{k} (-1)^k \nu^{\ast k} \hspace{1cm} (\ell \in \N_0).
\end{align}
Finally, convolution of $q_1, q_2 \in L^1(\R)$ is defined by $(q_1 \ast q_2)(\xi) := \int_\R q_1(\xi-y)q_2(y)dy$. The notion of convolution powers, however, can not directly be adopted to the space $L^1(\R)$, since $\delta_{\curbr{0}}$ is not absolutely continuous with respect to the Lebesgue measure.
\\
\hspace*{1em}The Fourier-Stieltjes transform of $F_\nu$, for $\nu \in \mathcal{M}(\K, \mathcal{B}(\R))$, is defined by the complex-valued integral
\begin{align} \label{DefcFX}
\Phi_\nu(t) := \intl_{-\infty}^\infty e^{itx} \nu(dx) \hspace{1cm} (t\in\R).
\end{align}
If $\nu$ is a probability measure, we refer to $\Phi_\nu$ as the c.f.. In any case, $\Phi_\nu(t)$ establishes a uniformly continuous function of $t \in \R$. Moreover, by means of inversion formulae, $F_\nu$ (as well as $f_\nu$, if existing) can be recovered from $\Phi_\nu$ (see also Appendix \ref{AppInvCF}). According to the product rule (or convolution theorem; compare. Theorem 3.3.1 in \cite{Lukacs1970}), $\Phi_{\nu_1 \ast \nu_2} = \Phi_{\nu_1}\Phi_{\nu_2}$, for all $\nu_1, \nu_2 \in \mathcal{M}(\K, \mathcal{B}(\R))$. Analogously, for an arbitrary function  $q : \R \rightarrow \C$, the integral
\begin{align} \label{DefFT}
\mathcal{F}\curbr{ q }(t) := \intl_{-\infty}^\infty e^{itx} q(x) dx \hspace{1cm} (t \in \R)
\end{align}
is simply referred to as the Fourier transform. It converges absolutely and uniformly with respect to $t \in \R$, whenever $q \in L^1(\R)$. If $\nu \in \mathcal{M}(\K, \mathcal{B}(\R))$ is absolutely continuous with respect to the Lebesgue measure, then $\mathcal{F}\curbr{ f_\nu } = \Phi_\nu$. Specifically for $x \mapsto \ONE_{[a, b]}(x)$, with $a<b$, we write
\begin{align} \label{cFUab}
\phi_{a,b}(t) := \mathcal{F}\curbr{\ONE_{[a, b]}}(t) = \frac{ e^{itb} - e^{ita} }{it} \hspace{1cm} (t \in \R).
\end{align}
The importance of c.fs. essentially lies in their existence for any kind of distributions, together with their unique invertibility. By construction, the c.f. of an arbitrary pro\-ba\-bi\-lity measure $\nu$ satisfies $\Phi_\nu(0) = 1$, $0 \leq \abs{ \Phi_\nu } \leq 1$ and $\overline{ \Phi_\nu(t)} = \Phi_\nu(-t)$, for all $t \in \R$. It is real-valued if and only if it is even, i.e., if $\Phi_\nu(t) = \Phi_\nu(-t)$, for all $t \in \R$. This is equivalent to symmetry of $\nu$ with respect to the origin, i.e., $F_\nu(\xi-) = 1-F_\nu(-\xi)$, for all $\xi \in \R$. The set of zeros in $\overline \R$ of $\Phi_\nu$ is referred to as
\begin{align*}
\mathcal{N}_\nu := \curbr{t \in \overline \R : \Phi_\nu(t) = 0 }.
\end{align*}
Lastly, due to the Lebesgue decomposition theorem (Theorem 1.1.3 in \cite{Lukacs1970}), there always exist $a_1, a_2, a_3 \geq 0$, with $\sum_{j=1}^3 a_j = 1$, such that
\begin{align*}
\Phi_\nu = a_1 \Phi_{\nu_d} + a_2 \Phi_{\nu_a} + a_3 \Phi_{\nu_s},
\end{align*}
where $\nu_d$ is a discrete, $\nu_a$ is an absolutely continuous and $\nu_s$ is a continuously singular probability measure. In particular, $\Phi_\nu$ corresponds to a pure distribution if $\max_{1 \leq j \leq 3}a_j=1$, and else it is a mixture. The single addends essentially can be distinguished by their behaviour at infinity:
\begin{itemize}
\item The discrete part $\Phi_{\nu_d}$ is a Fourier series, whose coefficients are the atoms of $\nu_d$. It is almost periodic in the sense of Bohr (see \cite{Bohr1932}) and satisfies $\limsup_{t \rightarrow \pm\infty} \abs{\Phi_{\nu_d}(t)} = 1$.
\item The absolutely continuous part fulfills $\Phi_{\nu_a} = \mathcal{F}\curbr{f_{\nu_a}}$. Thus, the Riemann-Lebesgue lemma applies, viz $\lim_{ t \rightarrow \pm\infty } |\Phi_{\nu_a}(t)| = 0$.
\item Regarding the singular part, $ \limsup_{t\rightarrow \pm\infty} |\Phi_{\nu_s}(t)| \in [0, 1]$, the exact superior limit depending on the distribution. Particularly if the superior limit equals zero, i.e., if $\Phi_{\nu_s}(t)$ vanishes as $t\rightarrow\pm\infty$, this needs to happen slower than the decay of any function of the space $L^1(\R)$. Else, it would contradict the inversion formula for densities (e.g., Theorem 3.2.2 in \cite{Lukacs1970}).
\end{itemize}

\section{Symmetrization of the deconvolution problem} \label{SecSymDecProb}

Symmetry plays a key role in many mathematical fields. The most frequently encountered examples are principal value integrals, with the partial sum operator in Fourier analysis as a special integral of that kind. It is not difficult to verify the divergence of such integrals without symmetry (see $\S$2.3.2 in \cite{Pinsky2002}). The importance of some kind of symmetry in the context of deconvolution will turn out in this section. In before, we show that a probability measure always can be transformed to a probability measure with a symmetric c.f. that ranges the unit interval. Notice that a general symmetric c.f. is not necessarily non-negative. Consider, for instance, the uniform distribution on $[-1,1]$, with c.f. $t \mapsto t^{-1} \sin(t)$.

\begin{lemma}[symmetrization of probability measures] \label{LemSymErr}
To any probability measure $\mu$, it exists a probability measure $\eta$, such that $\bareps := \mu \ast \eta$ has the symmetric d.f. $F_\bareps = F_\mu \ast F_\eta$ and the symmetric c.f. $\Phi_\bareps = \Phi_\mu \Phi_\eta$, with $0 \leq \Phi_\bareps \leq 1$ and $\mathcal{N}_\bareps = \mathcal{N}_\mu$.
\end{lemma}

\begin{proof}
The conjugate measure of an arbitrary probability measure $\mu$ is defined by $\eta(A) := \mu(\curbr{-a : a \in A})$, for $A \in \mathcal{B}(\R)$, and obviously again constitutes a probability measure. It has the d.f. $F_\eta(\xi) = 1-F_\eta(-\xi-)$, for $\xi \in \R$, and the c.f. is $\Phi_\eta(t)=\Phi_\mu(-t)$, for $t \in \R$. Therefore, the convolution $\mu\ast\eta$ again represents a probability measure, and the associated d.f. is $F_{\mu\ast\eta}(\xi) = \int_{-\infty}^\infty F_\mu(\xi+z) F_\mu(dz)$, for $\xi\in\R$, whereas the c.f. fulfills
\begin{align*}
\Phi_{\mu\ast\eta}(t) = \Phi_\mu(t)\Phi_\mu(-t) = \Phi_\mu(t)\overline{ \Phi_\mu(t) } = \abs{\Phi_\mu(t)}^2 \hspace{1cm} (t\in\R).
\end{align*}
It shows that $0 \leq \Phi_{\mu\ast\eta} \leq 1$, which completes the proof.
\end{proof}

We refer to $\bareps$ as a symmetrization of $\mu$ and emphasize that the symmetrizing factor is not unique. Instead, convolution with the conjugate is just a generally applicable strategy. As a matter of fact, since $\Phi_{\delta_{\curbr{0}}}\equiv1$, we can always find c.fs. $\Phi_{\mu_1}$ and $\Phi_{\mu_2}$, with $0 \leq \Phi_{\mu_1} \leq 1$, such that $\Phi_\mu=\Phi_{\mu_1}\Phi_{\mu_2}$. If $\Phi_{\mu_1}$ and $\Phi_{\mu_2}$ both are not associated with point measures, $\Phi_\mu$ is called decomposable or divisible (compare $\S$5 in \cite{Lukacs1970}). In any case, to obtain a symmetrization of $\mu$, convolution with the conjugate of the pro\-ba\-bility measure $\mu_2$ suffices. Particularly if $\mu_2 = \delta_{\curbr{c_0}}$, for a fixed $c_0 \in \R$, symmetrization is equivalent to centering.
\\
\hspace*{1em}Now, in the situation of Lemma \ref{LemSymErr}, also $\chi\ast\bareps$ is a probability measure. The associated d.f. satisfies the identity
\begin{align} \label{2026062601}
F_{\chi\ast\bareps} = F_\chi \ast F_\bareps = F_\chi \ast F_\mu \ast F_\eta = F_{\chi \ast \mu} \ast F_\eta.
\end{align}
Furthermore,
\begin{align} \label{FTsymY}
\Phi_{\chi\ast\bareps} = \Phi_\chi \Phi_\bareps.
\end{align}
For $t \in \R \setminus \mathcal{N}_\mu$, we may divide by $\Phi_\bareps(t)$ and eventually perform a geometric series expansion, since then $\abs{ 1- \Phi_\bareps(t) }< 1$, to arrive at a local representation for $\Phi_\chi(t)$. In order to be able to establish this observation in a definite theorem, in terms of the signed measure
\begin{align} \label{2026070401}
\pi_\bareps := \delta_{\curbr{0}}-\bareps,
\end{align}
we first introduce some additional definitions.

\begin{definition} [deconvolution kernel and function] \label{DefDm}
For $(\xi, m) \in \R \times \N_0$, we refer to
\begin{align} \label{Dekmkompakt}
\mathfrak{K}_\bareps(\xi, m) := \suml_{\ell=0}^m F_{\pi_\bareps}^{\ast \ell} (\xi)
\end{align}
as the deconvolution kernel, and the deconvolution function is given by
\begin{align} \label{2024120702}
\mathfrak{D}(\xi, m) := (F_{\chi \ast \bareps} \ast \mathfrak{K}_\bareps(\cdot, m)) (\xi).
\end{align}
\end{definition}

Specifically the deconvolution kernel, according to (\ref{Dekmkompakt}), represents a Neumann partial sum. Being composed of convolution powers, it slightly resembles a renewal function or a renewal measure from renewal theory, except for the fact that $F_{\pi_\bareps}$ is a signed d.f. rather than non-negative. In view of the binomial convolution identity (\ref{BinConvId}), corresponding convolution powers can be cast in the form
\begin{align} \label{Dekm}
F_{\pi_\bareps}^{\ast \ell} (\xi) = \suml_{k=0}^\ell \binom{\ell}{k} (-1)^k F_\bareps^{\ast k}(\xi) \hspace{1cm} ((\ell,\xi) \in \N_0 \times \R).
\end{align}
It shows that $F_{\pi_\bareps}^{\ast \ell}$ is the binomial transform (compare, e.g., exercise 36, p. 136 \cite{Knuth1998}) of the sequence of d.fs. $( F_\bareps^{\ast k} )_{k \in \N_0}$. Moreover, also $\mathfrak{K}_\bareps(\cdot, m)$ and $\mathfrak{D}(\cdot, m)$ are associated with signed measures, for any $m \in \N_0$, except $\mathfrak{K}_\bareps(\cdot, 0) = \IndNr{0 \, \leq \, \cdot}$ and $\mathfrak{D}(\cdot, 0) = F_{\chi \ast \bareps}$. Neumann sums are of frequent occurence in functional analysis, especially in the context of integral equations, where they are closely related to the so-called resolvent. Actually, also the deconvolution kernel can be put in such a context. For this, given $F_{\chi\ast\mu}$ and $F_\mu$, observe that the identity (\ref{Verteilungsfaltung}) can be conceived as a first kind integral equation for the unknown d.f. $F_\chi$. Straightforward manipulations lead to the second kind integral equation
\begin{align*}
F_\chi = F_{\chi \ast \bareps} + F_{\pi_\bareps} \ast F_\chi,
\end{align*}
which bears the advantage that $F_\chi$ can be approximated by means of Picard's ite\-ra\-tion. One can show that this approximation is exactly the deconvolution function $\mathfrak{D}(\cdot, m)$. In \cite{kaiser2025deconvolutiondistributionfunctionsintegral}, we discussed further transformations of the convolution identity (\ref{Verteilungsfaltung}) and properties of obtainable approximations in the domain of d.fs.. Here, we completely focus on the Fourier domain and therefore proceed with the derivation of the Fourier-Stieltjes transforms associated with the above functions. First, we introduce the $m$-power, that is
\begin{align} \label{mpwr}
\mathcal{P}_\bareps(t,m) := \rb{1-\Phi_\bareps(t)}^{m+1} \hspace{1cm} (t \in \R, ~m \geq 0).
\end{align}
Notice, according to the binomial theorem, for any $t_0 \in \mathcal{N}_\mu$, that
\begin{align} \label{2024122301}
\mathcal{P}_\bareps(t,m) = 1 + \LandauO\curbr{ \Phi_\bareps(t) } \hspace{1cm} (t \rightarrow t_0).
\end{align}
Moreover, denoting the Fourier-Stieltjes transform of the deconvolution kernel by
\begin{align} \label{GeomPartialSummenFkt01A}
\mathcal{G}_\bareps(t,m) := \intl_{-\infty}^\infty e^{itz} \mathfrak{K}_\bareps(dz, m) \hspace{1cm} ( (t, m) \in \R\times \N_0),
\end{align}
with the aid of (\ref{Dekmkompakt}) and (\ref{Dekm}), through a simple application of the binomial and geometric sum formulae, it is easy to show that
\begin{subequations} \label{GeomPartialSummenFkt01}
\begin{align} \label{GeomPartialSummenFkt01B}
\mathcal{G}_\bareps(t,m) &= \suml_{\ell=0}^m \rb{ 1 - \Phi_\bareps(t)}^\ell
\\ \label{GeomPartialSummenFkt02}
&=
\begin{cases}
\frac{1-\mathcal{P}_\bareps(t,m)}{\Phi_\bareps(t)}, & \mbox{for } t \in\R\setminus \mathcal{N}_\mu, \\
m+1, & \mbox{for } t \in \mathcal{N}_\mu.
\end{cases}
\end{align}
\end{subequations}
We refer to $\mathcal{G}_\bareps(t,m)$ as the geometric sum function. Informally speaking, it constitutes the $m$-th partial sum of a geometric series, with $\lim_{m\rightarrow\infty} \mathcal{G}_\bareps(t,m) = \curbr{ \Phi_\bareps(t) }^{-1}$, for $t\in\R\setminus \mathcal{N}_\mu$. Thereof, since $0\leq\Phi_\bareps\leq1$, we conclude that this limit never constitutes a c.f., unless $\Phi_\bareps \equiv 1$. Finally, we denote the Fourier-Stieltjes transform associated with the deconvolution function by
\begin{align} \label{FTDm1A}
\Phi_\mathfrak{D}(t,m) := \intl_{-\infty}^\infty e^{itx} \mathfrak{D}(dx, m) \hspace{1cm} ( (t, m) \in \R \times \N_0 ).
\end{align}
Then, in view of (\ref{2024120702}) and (\ref{GeomPartialSummenFkt01}), for $(t, m) \in \R \times \N_0$, we compute
\begin{subequations} \label{FTDm1}
\begin{align} \label{FTDm1B}
\Phi_\mathfrak{D}(t,m) &= \Phi_{\chi \ast \bareps}(t) \mathcal{G}_\bareps(t,m)
\\ \label{FTDm2}
&= \Phi_\chi(t)\curbr{1-\mathcal{P}_\bareps(t,m)}.
\end{align}
\end{subequations}
Our first theorem summarizes a few convergence properties of $\Phi_\mathfrak{D}(\cdot, m)$ and is thus an early indicator for the importance of the deconvolution function.

\begin{theorem}[properties of $\Phi_\mathfrak{D}$] \label{Theo:DeccF}
The Fourier-Stieltjes transform of $\mathfrak{D}(\cdot,m)$ fulfills
\begin{align} \label{SchrankeFTDm}
\norm{ \Phi_\mathfrak{D}(\cdot, m) }_\infty \leq 1 \hspace{1cm} (m \geq 0),
\end{align}
and it exhibits the following convergence behaviour:
\begin{enumerate}[label=(\arabic*),ref=\thelemma~(\arabic*)]
\item\label{Theo:DeccF:1} If $t \in \R\setminus \mathcal{N}_\mu$ or $t \in \mathcal{N}_\mu \cap \mathcal{N}_\chi$, we have
\begin{align*}
\liml_{m \rightarrow \infty} \abs{ \Phi_\mathfrak{D}(t,m) - \Phi_\chi(t) } = 0.
\end{align*}
The convergence is uniform on any compact interval $I\subset\R$ with $I\cap \mathcal{N}_\mu \subseteq \mathcal{N}_\chi$.
\item\label{Theo:DeccF:2} Provided $\mathcal{N}_\mu \subseteq \mathcal{N}_\chi$ and $\pm\infty \in \mathcal{N}_\chi$, then
\begin{align*}
\liml_{m \rightarrow \infty} \norm{ \Phi_\mathfrak{D}(\cdot,m) - \Phi_\chi }_\infty = 0.
\end{align*}
\end{enumerate}
\end{theorem}

The uniform convergence on the whole real axis is non-trivial and will only occur if the sequence $\norm{ \Phi_\mathfrak{D}(\cdot,m) - \Phi_\chi }_\infty$ is bounded away from unity, for all sufficiently large $m$. Moreover, the characterization of this type of convergence by the above theorem is incomplete. For example, consider the c.f. $\Phi_\bareps(t) := \frac{1}{2} \curbr{ \cos(t) }^2 + \frac{1}{2} \exp\curbr{-t^2}$, associated with a mixture distribution. Then, $\mathcal{N}_\mu = \emptyset$. Besides, for $k \in \N_0$ and $t_k := (2k+1)\frac{\pi}{2}$, we have $\mathcal{P}_\bareps(t_k,m) = (1-\frac{1}{2} \exp\curbr{-t_k^2})^{m+1} \rightarrow 1$, as $k\rightarrow \infty$. Hence, $\norm{ \mathcal{P}_\bareps(\cdot, m) }_\infty = 1$, for $m \geq 0$. However, if additionally $\Phi_\chi(t) := \cos(t)$, then $\Phi_\mathfrak{D}(t_k,m) = \Phi_\chi(t_k) = 0$, for each $k \in \N_0$, and still $\norm{ \Phi_\mathfrak{D}(\cdot,m) - \Phi_\chi }_\infty \rightarrow 0$, as $m \rightarrow \infty$.

\begin{proof}[Proof of Theorem \ref{Theo:DeccF}]
The uniform boundedness (\ref{SchrankeFTDm}) is an immediate consequence of the representation (\ref{FTDm2}), since $\Phi_\chi$ and $\Phi_\bareps$ are also uniformly bounded. According to this representation, we also get
\begin{align} \label{2024120712}
\Phi_\chi(t) - \Phi_\mathfrak{D}(t,m) = \Phi_\chi(t) \mathcal{P}_\bareps(t,m) \hspace{1cm} (t \in \R, ~ m \geq 0).
\end{align}
Therefore, $\abs{ \Phi_\chi(t) - \Phi_\mathfrak{D}(t,m)} < 1$, for $t \in \R\setminus \mathcal{N}_\mu$, and the modulus equals zero if even $t \in \mathcal{N}_\chi$. The monotonicity of $\mathcal{P}_\bareps(t,m)$ with respect to $m\geq0$, for $t \in\R \setminus \mathcal{N}_\mu$, thus implies the asserted pointwise convergence. The uniformity on any compact subset is then merely a consequence of Dini's theorem, by continuity of (\ref{2024120712}) and by continuity of the limit function, valid under the assumption $\Phi_\chi(t) = 0$, for every $t\in I\cap \mathcal{N}_\mu$. To eventually verify Theorem \ref{Theo:DeccF:2} we note, since $\lim_{\abs{t}\rightarrow\infty}\Phi_\chi(t) = 0$, for any $\delta>0$, that there exists $R>0$ with
\begin{align*}
\supl_{\abs{t}>R} \abs{ \Phi_\chi(t) \mathcal{P}_\bareps(t,m) } \leq \supl_{\abs{t}>R} \abs{ \Phi_\chi(t) } < \delta,
\end{align*}
for all $m \geq 0$. In view of Theorem \ref{Theo:DeccF:1}, however, the convergence on $[-R,R]$ is uniform.
\end{proof}

Owing to the fact that $\mathfrak{D}(\cdot, m)$ is not associated with a non-negative measure, the continuity theorem for c.fs. (Theorem 3.6.1 in \cite{Lukacs1970}) is inapplicable, and Theorem \ref{Theo:DeccF} does not imply the convergence to $F_\chi$, i.e., weak convergence. To verify this convergence for a large class of distributions, and thereby justify the applicability of the deconvolution function for the reconstruction of $F_\chi$, will be the subject of \S\ref{ChConvDecFct}. In before, we present further supplementary results.

\section{Elementary properties of the deconvolution function} \label{SubSecDekma}

Since the Dirac measure with mass at the origin is associated with the neutral element of convolution, it is obvious that $\chi \ast \bareps = \chi$ if and only if $\bareps = \delta_{ \curbr{0} }$. Conversely, $\chi \ast \bareps \neq \chi$, whenever $\bareps \neq \delta_{ \curbr{0} }$. Therefore, the signed measure $\pi_\bareps$, compare (\ref{2026070401}), characterizes the deviation of $\bareps$ from $\delta_{ \curbr{0} }$. With regard to the associated d.f. $F_{\pi_\bareps}$, the binomial identity (\ref{Dekm}) shows that arbitrary convolution powers satisfy
\begin{align} \label{DekMa03}
\liml_{\xi \rightarrow \pm\infty} F_{\pi_\bareps}^{\ast \ell}(\xi) = 0 \hspace{1cm} (\ell \in \N).
\end{align}
As an immediate consequence, deconvolution kernel and function both asymptotically exhibit the typical behaviour of a d.f., viz
\begin{align} \label{DekMa06}
\liml_{\xi \rightarrow \xi_0} \mathfrak{K}_\bareps(\xi, m) = \liml_{\xi\rightarrow \xi_0} \mathfrak{D}(\xi, m) =
\begin{cases}
0, & \mbox{if } \xi_0 = -\infty, \\
1, & \mbox{if } \xi_0 = \infty.
\end{cases}
\end{align}
Definition \ref{DefDm} moreover suggests that the deconvolution kernel $\mathfrak{K}_\bareps(\xi, m)$ is rarely con\-ti\-nuous at $\xi = 0$, due to the presence of the Dirac d.f. $\IndNr{0 \, \leq \, \xi}$ in at least one summand, for each $m \in \N_0$. The graphs in Figure \ref{deconvolution_kernel.fig} support this guess. These also illustrate the symmetry with respect to $\xi$ and the unbounded behaviour of $\mathfrak{K}_\bareps(\xi, m)$ as $m$ increases, which is the subject of our next lemma.

\begin{lemma} \label{Lemma2026062601}
If $\bareps$ is continuous, for all $\xi \in \R$, it holds that
\begin{align} \label{DekMa05}
F_{\pi_\bareps}^{\ast \ell} (\xi) = - F_{\pi_\bareps}^{\ast \ell}(-\xi) \ONE_{\R \setminus \curbr{0}}(\xi) + \frac{1}{2} \IndNr{0}(\xi) \hspace{1cm} (\ell \in \N),
\end{align}
as well as that
\begin{align} \label{DekMa07}
\mathfrak{K}_\bareps(\xi, m) = (1 - \mathfrak{K}_\bareps(-\xi, m) ) \ONE_{\R \setminus \curbr{0}}(\xi) + \frac{m+2}{2} \IndNr{0}(\xi) \hspace{1cm} (m \in \N_0).
\end{align}
\end{lemma}

\begin{proof}
By continuity and symmetry of $\bareps$, we first observe that $F_\bareps^{\ast k}(\xi) = 1- F_\bareps^{\ast k}(-\xi)$ and $F_\bareps^{\ast k}(0) = \frac{1}{2}$, for all $(k, \xi) \in \N \times \R\setminus\curbr{0}$, whereas $F_\bareps^{\ast 0}(\xi) = 1 - F_\bareps^{\ast 0}(-\xi)$ and $F_\bareps^{\ast 0}(0) = 1$. The identity (\ref{DekMa05}) is thus a consequence of the expansion (\ref{Dekm}). Thereof, with regard to (\ref{DekMa07}), for $(\xi, m) \in \R \setminus \curbr{0} \times \N_0$, we get $\mathfrak{K}_\bareps(\xi, m) = 1 - \IndNr{0 \, \leq \, -\xi} - \sum_{\ell=1}^m F_{\pi_\bareps}^{\ast \ell} (-\xi)$ and $\mathfrak{K}_\bareps(0, m) = \frac{m+2}{2}$, which completes the proof.
\end{proof}

\begin{figure}
\centering
\resizebox*{\textwidth}{!}{\includegraphics{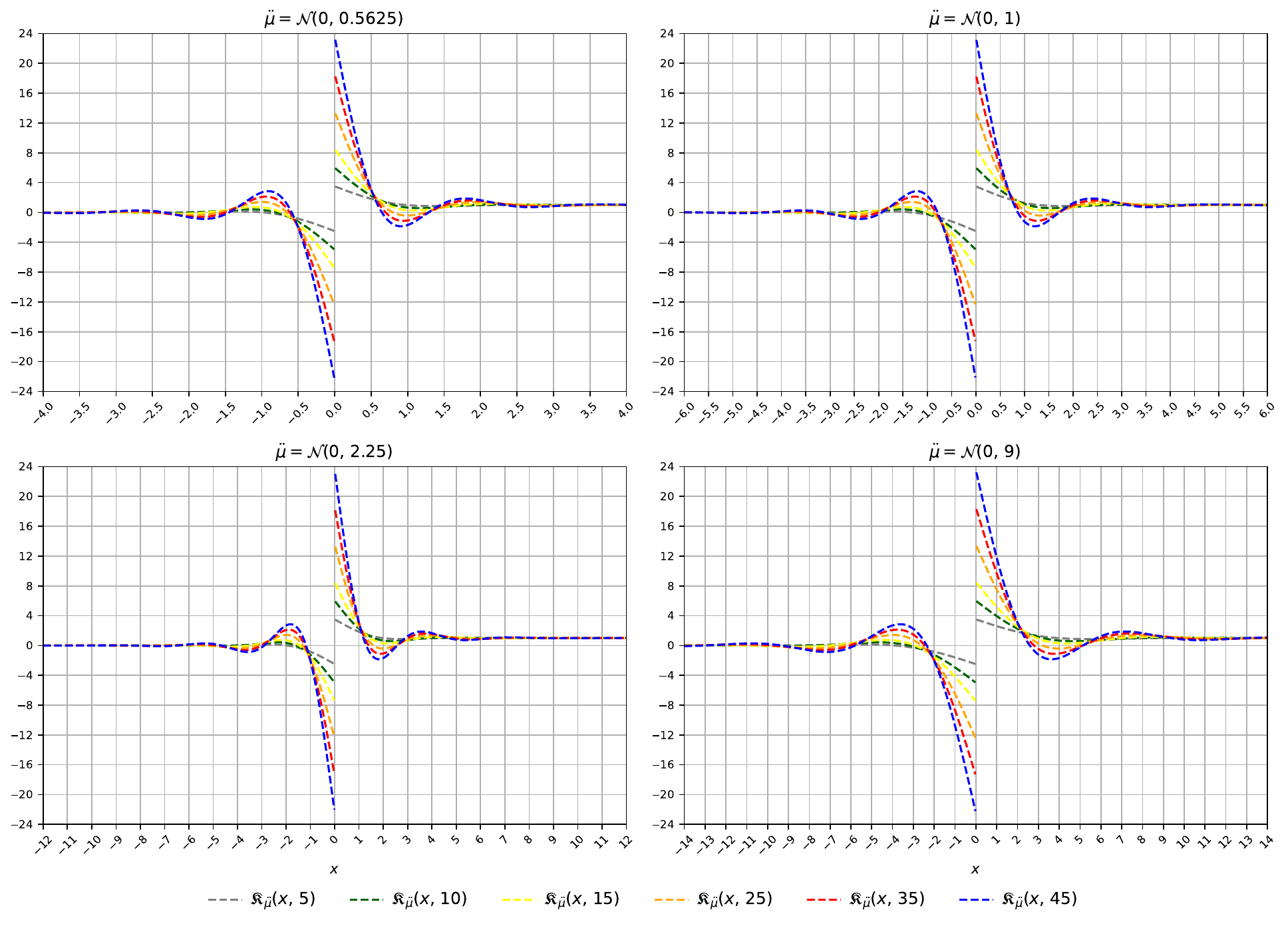}}
\vspace{-0.5cm}
\caption{Plots for deconvolution kernels $\mathfrak{K}_\bareps(\cdot, m)$, where $\bareps$ corresponds to various Gauss distributions with increasing variances. Notice the discontinuity at $\xi=0$ and the growth there.} \label{deconvolution_kernel.fig}
\end{figure}

By additional convolution with $\bareps$, the deconvolution kernel simplifies analogous to the well-known geo\-me\-tric identity $q\sum_{\ell=0}^m (1-q)^\ell = 1-(1-q)^{m+1}$, for $|q|<1$.

\begin{lemma} \label{BasicPropDecfct}
We have
\begin{align*}
F_\bareps \ast \mathfrak{K}_\bareps(\cdot, m) = \IndNr{0 \, \leq \, \cdot} - F_{\pi_\bareps}^{\ast(m+1)} \hspace{1cm} (m \in \N_0).
\end{align*}
\end{lemma}

\begin{proof}
The asserted identity becomes obvious upon writing
\begin{align*}
\mathfrak{K}_\bareps(\cdot, m) \ast F_\bareps = \mathfrak{K}_\bareps(\cdot, m) \ast ( \IndNr{0 \, \leq \, \cdot} - F_{\pi_\bareps} ),
\end{align*}
and representing the right hand side in terms of the Neumann sum (\ref{Dekmkompakt}).
\end{proof}

Observe that $\mathfrak{K}_\bareps(\cdot, m) \ast F_\bareps$ is continuous, if $\bareps$ is continuous. In particular, $\mathfrak{K}_\bareps(\cdot, m) \ast F_\bareps = \mathfrak{D}(\cdot, m)$, if $\chi=\delta_{ \curbr{0} }$, according to (\ref{2024120702}). This situation deserves a special emphasis, since the deconvolution function then corresponds to an approximate identity, i.e., it represents an approximation for the identity of the convolution of d.fs.. Figure \ref{deconvolution_binomial.fig} visualizes such a setup and indicates some kind of convergence, to be verified in a later section. Inte\-res\-tingly, the approximating functions $\mathfrak{D}(\xi, m)$ in these examples are continuous with respect to $\xi\in\R$, whereas $F_\chi(\xi)=\ONE_{\curbr{0 \leq \xi}}$ has a jump at $\xi=0$. The obvious question is therefore, what happens at $\xi=0$, as $m \rightarrow \infty$. Now, if $\bareps$ is an arbitrary continuous symmetric probability measure, we know from Lemma \ref{BasicPropDecfct} and (\ref{DekMa05}) that
\begin{align} \label{DekMa09}
(F_\bareps \ast \mathfrak{K}_\bareps(\cdot, m)) (0) = \IndNr{0 \, \leq \, 0} - F_{\pi_\bareps}^{\ast (m+1)} (0) = \frac{1}{2} \hspace{1cm} (m\in\N_0).
\end{align}
Hence, if $\chi =  \delta_{ \curbr{0} }$, it turns out that
\begin{align} \label{DekMa010}
\liml_{m \rightarrow \infty} \mathfrak{D}(0,m) = \frac{1}{2} \neq 1 = F_\chi(0).
\end{align}
The value $\mathfrak{D}(0,m) = \frac{1}{2}$ is simply the mean of the left- and right-sided limits of $F_\chi(\xi)$ at the discontinuity $\xi=0$. Such a behaviour is very common in the inversion of integral transforms (see \S\ref{SecFourierIntDecFct} and Appendix \ref{AppInvCF}). In summary, we have seen that discontinuities of the target d.f. $F_\chi$ require special care. The previous example has furthermore shown that the properties of the deconvolution function $\mathfrak{D}(\cdot, m)$ substantially differ from those of the kernel, due to the additional convolution with $F_{\chi\ast\bareps}$. In particular, $\mathfrak{D}(\cdot, m)$ thereby inherits continuity properties, if $\chi$ or $\bareps$ is continuous. It is even absolutely continuous, if $\chi$ or $\bareps$ is absolutely continuous. We then refer to
\begin{align} \label{DecDens1}
\mathfrak{d}(\xi,m) := \intl_{-\infty}^\infty f_{\chi \ast \bareps}(\xi-x) \mathfrak{K}_\bareps(dx, m) \hspace{1cm} ( (\xi, m) \in \R \times \N_0 )
\end{align}
as the deconvolution density, and the following statement holds.

\begin{lemma} \label{LemBinDensTrans}
If $F_\chi$ or $F_\bareps$ is absolutely continuous, then $\mathfrak{D}(\xi,m)$ is differentiable at Lebesgue almost every $\xi \in \R$, with derivative $\mathfrak{D}'(\xi, m) = \mathfrak{d}(\xi,m)$, for any $m \in \N_0$. In particular, $\int_{-\infty}^\xi \mathfrak{d}(x,m) dx = \int_{-\infty}^\xi \mathfrak{D}(dx, m)$, for all $\xi \in \R$, and $\int_{-\infty}^\infty \mathfrak{d}(x,m) dx = 1$.
\end{lemma}

\begin{proof}
In the described situation, it follows by construction and from Theorem 3.3.2 in \cite{Lukacs1970}, that $\mathfrak{D}(\cdot,m)$ is absolutely continuous. Therefore, the lemma is a simple consequence of the Lebesgue differentiation theorem.
\end{proof}

The deconvolution density can be considered as an approximation for $f_\chi$, if ex\-is\-ting. Yet, since $\mathfrak{D}(\cdot, m)$ is associated with a signed measure, $\mathfrak{d}(\cdot,m)$ does not constitute a probability density, i.e., it is not non-negative. We conclude this section by showing that any finite moment of the deconvolution function, denoted by
\begin{align*}
\Mu_\mathfrak{D}(k,m) := \intl_{-\infty}^\infty x^k \mathfrak{D}(dx,m) \hspace{1cm} ( (k, m) \in \N_0^2 ),
\end{align*}
matches the corresponding moment of $F_\chi$, as $m\rightarrow\infty$.

\begin{theorem}[moments of $\mathfrak{D}(\cdot, m)$] \label{SectMomDecfct}
Suppose the existence of $K_\chi \in \N_0$ and $K_\bareps \in \N \setminus \curbr{1}$, such that $F_\chi$ and $F_\bareps$ have moments $\Mu_\chi(j)$ and $\Mu_\bareps(k)$, for every $0 \leq j \leq K_\chi$ and $0 \leq k \leq K_\bareps$. Define $K_0 := \min\rrb{K_\chi, K_\bareps}$. Then, for all $0 \leq k < \min\curbr{K_0+1, 2(m+1)}$,
\begin{align} \label{MomDecFct7}
\Mu_\mathfrak{D}(k,m) = \Mu_\chi(k).
\end{align}
In particular, $\Mu_\mathfrak{D}(2(m+1),m) \neq \Mu_\chi(2(m+1))$, if $2(m+1) \leq K_0$.
\end{theorem}

\begin{proof}
For $0\leq k \leq K_0$ and $\ell\in\N_0$, appealing to the multinomial theorem, it is easy to verify that $\Mu_{ \bareps^{\ast\ell} }(k) < \infty$ and $\Mu_{\chi \ast \bareps^{\ast\ell} }(k) < \infty$, i.e., that $F_\bareps^{\ast \ell}$ and $F_\chi \ast F_\bareps^{\ast \ell}$ both have moments up to order $K_0$. Notice that $\Mu_{\chi \ast \bareps^{\ast0} }(k) = \Mu_\chi(k)$. Since $\Phi_\chi\Phi_\bareps^\ell$ constitutes the c.f. associated with $F_\chi \ast F_\bareps^{\ast \ell}$, Corollary 2 to Theorem 2.3.1 in \cite{Lukacs1970} tells us that this function may be differentiated $k$-times, with
\begin{align} \label{MomDecFct1}
\Mu_{\chi \ast \bareps^{\ast\ell} }(k) = i^{-k} \cb{\frac{d^k}{dt^k} \Phi_\chi(t) \curbr{\Phi_\bareps(t)}^\ell}_{t=0} \hspace{1cm} (0 \leq k \leq K_0).
\end{align}
In addition, under the current assumptions, also $\Mu_\mathfrak{D}(k,m)$, for $0 \leq k \leq K_0$, exists. Therefore, according to (\ref{2026062601}), (\ref{2024120702}), Lemma \ref{BasicPropDecfct} and (\ref{MomDecFct1}), in terms of the $m$-power (\ref{mpwr}), we get
\begin{align} \label{MomDecFct3}
\Mu_\mathfrak{D}(k,m) = \Mu_\chi(k) - i^{-k} \cb{\frac{d^k}{dt^k} \Phi_\chi(t) \mathcal{P}_\bareps(t,m)}_{t=0}.
\end{align}
Due to symmetry, any odd moment of $F_\bareps$ equals zero, so that, by Theorem 2.3.3 in \cite{Lukacs1970}, an expansion of the form
\begin{align} \label{MomDecFct4}
\Phi_\bareps(t) = 1 + \suml_{ j = 1 }^{\big\lfloor\frac{K_\bareps}{2}\big\rfloor} c_{2j} (it)^{2j} + o \rrb{ t^{K_\bareps} } \hspace{1cm} (t \rightarrow 0)
\end{align}
holds, where $c_{2j} := (2j!)^{-1} \Mu_\bareps(2j)$, and the sum is non-empty, by assumption. Now, introducing
\begin{align*}
\rho(t,m) := \Phi_\chi(t) \frac{\mathcal{P}_\bareps(t,m)}{t^{2(m+1)}},
\end{align*}
we recast (\ref{MomDecFct3}), to arrive at
\begin{align} \label{MomDecFct5A}
\Mu_\mathfrak{D}(k,m) = \Mu_\chi(k) - i^{-k} \cb{\frac{d^k}{dt^k} t^{2(m+1)}\rho(t,m) }_{t=0}.
\end{align}
The function $\rho(t,m)$ is $K_0$-times differentiable and, by (\ref{MomDecFct4}), it satisfies $\rho(0,m) = c_2^{m+1}$. Moreover, for $0\leq k \leq \min\rrb{K_0,2(m+1)}$, it is obvious from the product rule of differentiation, as $t \rightarrow 0$, that
\begin{align} \label{MomDecFct6}
\frac{d^k}{dt^k} t^{2(m+1)}\rho(t,m) &= \frac{(2(m+1))!}{(2(m+1)-k)!} t^{2(m+1)-k}\rho(t,m) + o\rrb{t^{2(m+1)-k}}.
\end{align}
Accordingly, (\ref{MomDecFct5A}) cancels to (\ref{MomDecFct7}), for each $0 \leq k < \min\rrb{K_0+1,2(m+1)}$. Finally, if $2(m+1) \leq K_0$, from (\ref{MomDecFct5A}) and (\ref{MomDecFct6}), we deduce that
\begin{align*}
\Mu_\chi(2(m+1)) - \Mu_\mathfrak{D}(2(m+1),m) = i^{-2(m+1)} (2(m+1))! c_2^{m+1} \neq 0,
\end{align*}
which completes the proof.
\end{proof}

\begin{figure}
\centering
\resizebox*{\textwidth}{!}{\includegraphics{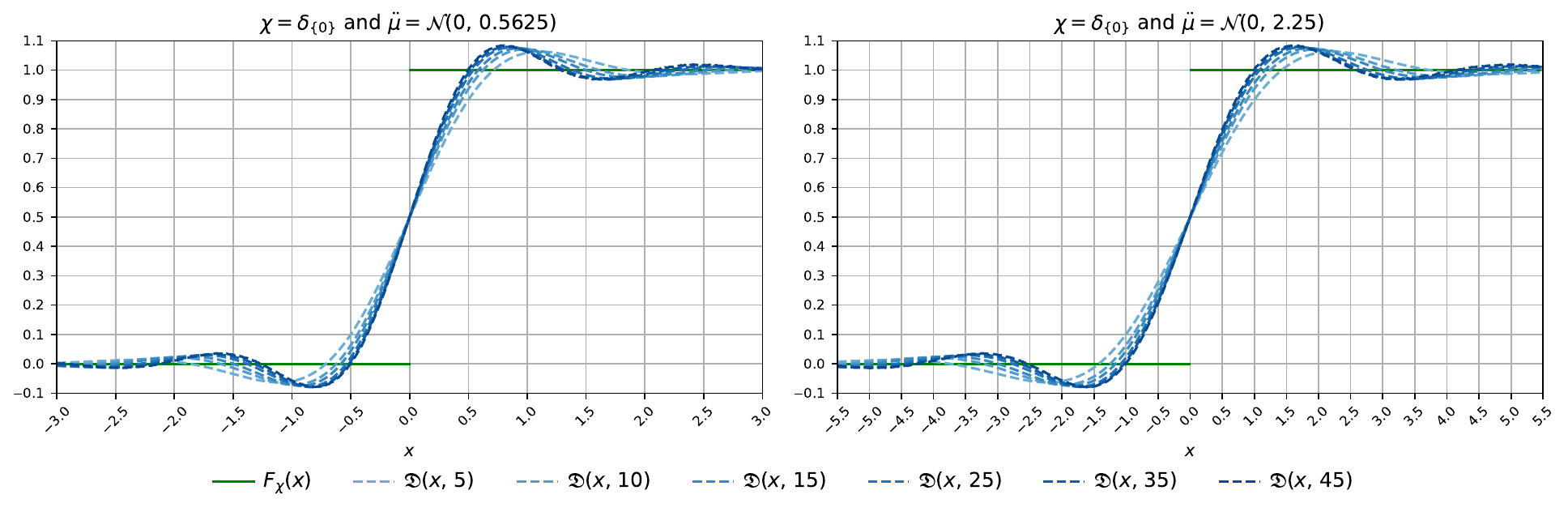}}
\vspace{-0.5cm}
\caption{Plots for the deconvolution function as an approximation for the identity of convolution. Notice that we approximate in each case a discontinuous function by a sequence of continuous functions.} \label{deconvolution_binomial.fig}
\end{figure}

\section{Integral representations for the deconvolution function} \label{SecFourierIntDecFct}

The representation of the deconvolution function as a Neumann sum is problematic for both, theoretical investigations and numerical evaluation. The combination of convolution powers and binomial coefficients, already for small numbers, inflicts computational inaccuracies and errors, as the limit of capacity is reached. As an alternative, in this section, we eventually invoke some representations as a Fourier-type integral, which basically follow from the inversion formulae from Appendix \ref{AppInvCF}.

\begin{lemma}[Fourier-type integral for $\mathfrak{D}(\cdot,m)$] \label{IntegraldarstDm}
For $\xi \in \R$, we have
\begin{align*}
\frac{\mathfrak{D}(\xi, m)+\mathfrak{D}(\xi-, m)}{2} = \frac{1}{2} + \liml_{ \substack{ T_1 \downarrow 0 \\ T_2 \uparrow \infty } } \intl_{T_1}^{T_2} \frac{e^{i\xi t} \Phi_{\chi \ast \bareps}(-t) - e^{-i\xi t} \Phi_{\chi \ast \bareps}(t)}{ i2\pi t} \mathcal{G}_\bareps(t,m) dt,
\end{align*}
the order of the limits being arbitrary.
\end{lemma}

\begin{proof}
By construction, the deconvolution function represents the d.f. of the signed measure defined by $\mu_\mathfrak{D}(E, m) := \int_E \mathfrak{D}(dx, m)$, for $E \in \mathcal{B}(\R)$. In particular, $\mu_\mathfrak{D}(\cdot, m) \in \mathcal{M}(\R, \mathcal{B}(\R))$. The associated Fourier-Stieltjes transform $\Phi_\mathfrak{D}(\cdot, m)$ was determined in (\ref{FTDm1}). Hence, the asserted representation is a direct consequence of the inversion formula from Theorem \ref{SatzInvHalbstr}, due to the evenness of $\mathcal{G}_\bareps(t, m)$.
\end{proof}

The identity from Lemma \ref{IntegraldarstDm} simplifies for continuity points $\xi \in \R$, since there $\mathfrak{D}(\xi, m)=\mathfrak{D}(\xi-, m)$. With regard to discontinuities, an inspection of the Neumann sum for $\mathfrak{D}(\cdot, m)$ shows that $D_{ \mathfrak{D}(\cdot, m) } \subset D_{ \mathfrak{D} }$, for each $m \in \N_0$, where
\begin{align} \label{UnstetMengDm}
D_{ \mathfrak{D} } := \rrb{ \xi \in \R : \mbox{ it exists } j\in \N \mbox{ with } (F_\chi \ast F_\bareps^{\ast j})\curbr{ \xi } > 0 }.
\end{align}
It can be a serious challenge to determine these discontinuities. In any case, $\mathfrak{D}(\cdot, m)$ is continuous on $\R \setminus D_{ \mathfrak{D} }$, for each $m \in\N_0$. Obviously, the most convenient situation corresponds to continuous $F_\bareps$ or $F_\chi$, in which $D_{ \mathfrak{D} } = \emptyset$. Now, an important means to quantify the deviation of $\mathfrak{D}(\cdot, m)$ from the target $F_\chi$ is the bias. The associated Fourier-type integral can be obtained similar to Lemma \ref{IntegraldarstDm}. To avoid the appearance of two limits, we require the following.

\begin{assumption} \label{Ass2026070901}
It exists $\tau > 0$, with
\begin{align*}
\intl_0^\tau \frac{\mathcal{P}_\bareps(t,0)}{t} dt < \infty.
\end{align*}
\end{assumption}

Since $\Phi_\bareps(0)=1$ always, it can be expected that Assumption \ref{Ass2026070901} holds for nume\-rous c.fs., yet, apparently it is not naturally fulfilled. We could not find a counterexample. A sufficient condition is that $\Phi_\bareps(t) = 1 + \LandauO(t^b)$, as $t \downarrow 0$, for $b>0$. Generally speaking, the behaviour of a c.f. near the origin depends on the tail behaviour of the d.f..

\begin{lemma}[Fourier-type integrals for the bias of $\mathfrak{D}(\cdot,m)$] \label{IntegraldarstBiasDm}
Under Assumption \ref{Ass2026070901}, the local bias at $\xi \in \R$ is given by
\begin{align} \label{BiasDekfkt}
\frac{\mathfrak{D}(\xi, m)+\mathfrak{D}(\xi-, m)}{2} - \frac{ F_\chi(\xi) + F_\chi(\xi-) }{2} &= \liml_{T \rightarrow \infty} \mathfrak{I}_T(m,\xi),
\end{align}
where we write
\begin{align} \label{PointwConv01b}
\mathfrak{I}_T(m,\xi) := \frac{1}{2\pi i} \intl_{-T}^T \frac{\mathcal{P}_\bareps(t,m)}{t}  e^{-i\xi t} \Phi_\chi(t) dt.
\end{align}
\end{lemma}

\begin{proof}
By definition of $\mathfrak{D}(\cdot, m)$, compare (\ref{2024120702}), and Lemma \ref{BasicPropDecfct}, for each $m \geq 0$, we can write $\mathfrak{D}(\cdot, m) - F_\chi = - F_\chi\ast F_{\pi_\bareps}^{\ast (m+1)}$. Apply the inversion formula from Theorem \ref{SatzInvHalbstr} to this d.f. and rearrange the integral, to arrive at (\ref{BiasDekfkt}).
\end{proof}

\begin{figure}
\centering
\resizebox*{\textwidth}{!}{\includegraphics{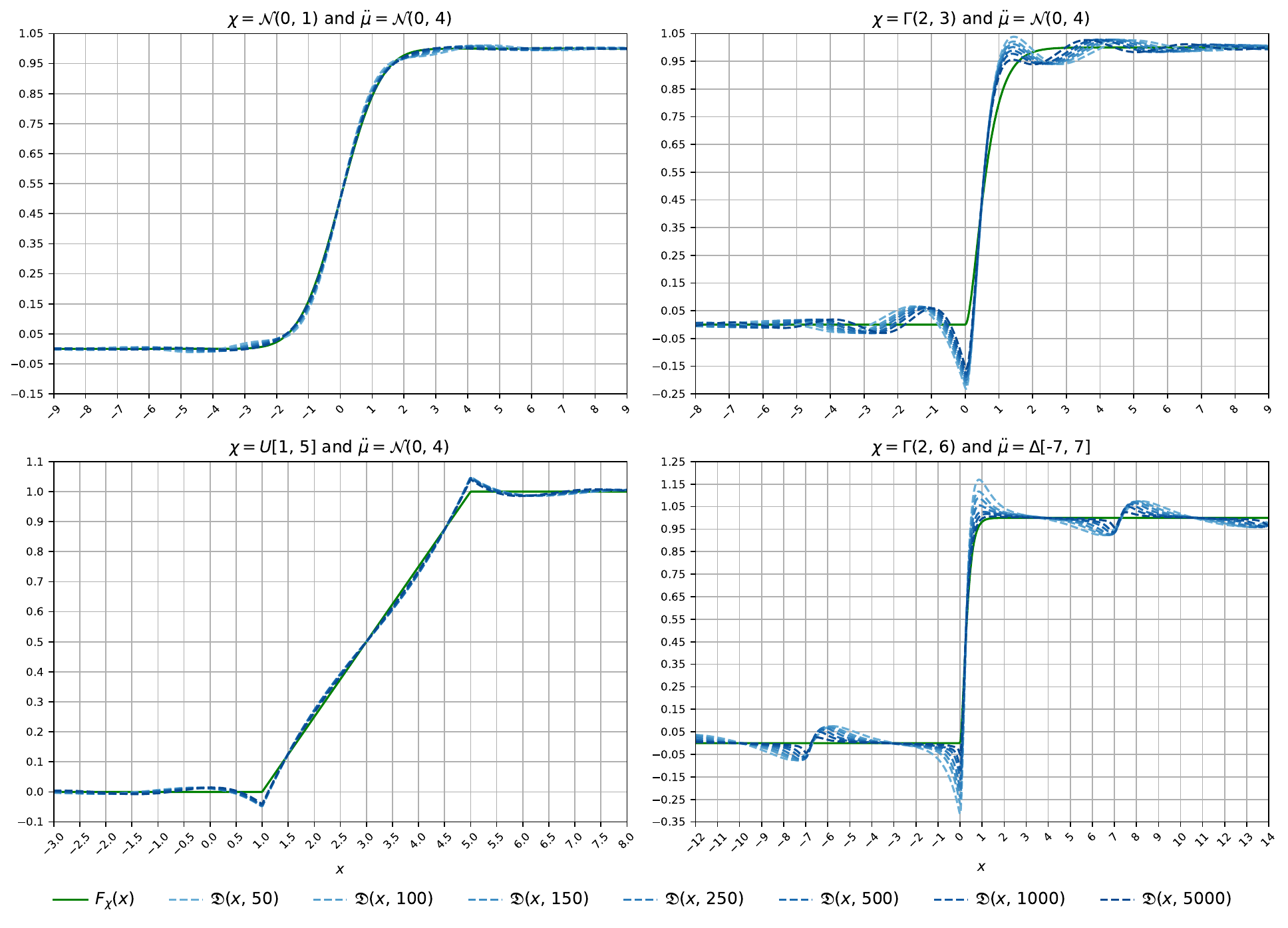}}
\vspace{-0.5cm}
\caption{The deconvolution function for various absolutely continuous distributions.} \label{df_deconvolution_fourier.fig}
\end{figure}

Figure \ref{df_deconvolution_fourier.fig} illustrates the deconvolution function for various absolutely continuous probability measures. These plots were created with the aid of the Fourier integral from Lemma \ref{IntegraldarstBiasDm}, which especially facilitates the evaluation for large values of $m$. Finally, a widely permissible representation for the deconvolution density as an integral of Fourier-type can not be established without use of an auxiliary function.

\begin{assumption}[smoothing kernel] \label{Ass2024121902}
$\iota : \mathcal{B}(\R) \rightarrow [0,1]$ is an absolutely continuous probability measure, with $f_\iota(\xi) = \LandauO\curbr{ \xi^{-2} }$, as $\xi \rightarrow \pm\infty$, and $\Phi_\iota \in L^1(\R)$.
\end{assumption}

The smoothing kernel is crucial for the transition from the domain of densities to Fourier transforms.

\begin{lemma}[Fourier-type integral for $\mathfrak{d}(\cdot,m)$] \label{LemFourIntDecDens}
Suppose that $\chi \ast \bareps$ is absolutely continuous, with continuous $f_{\chi \ast \bareps}$, and that Assumption \ref{Ass2024121902} holds. Then,
\begin{align*}
\mathfrak{d}(\xi,m) &= \frac{1}{2\pi} \liml_{\delta \downarrow 0} \intl_{-\infty}^\infty e^{-i\xi t} \Phi_\iota(\delta t) \Phi_\mathfrak{D}(t, m) dt \hspace{1cm} (\xi \in \R, ~ m \geq 0).
\end{align*}
\end{lemma}

Notice that the limit can be performed under the sign of integration, if $\Phi_\mathfrak{D}(\cdot, m) \in L^1(\R)$. Moreover, since $f_{\chi \ast \bareps} = f_\chi \ast f_\bareps$, it suffices that one of the components $f_\chi$ or $f_\bareps$ is continuous, in order to have continuity of $f_{\chi \ast \bareps}$.

\begin{proof}[Proof of Lemma \ref{LemFourIntDecDens}]
Under the current assumptions, in view of (\ref{DecDens1}), the deconvolution density exists and represents a continuous function $\mathfrak{d}(\cdot, m) \in L^1(\R)$. The associated Fourier transform is $\Phi_\mathfrak{D}(\cdot, m)$. The given formula is hence an immediate consequence of Theorem \ref{TheoInvDens}.
\end{proof}

In an analogous fashion, we eventually also represent the bias of the deconvolution density as a Fourier-type integral. For this, of course, the actual existence of the target density is necessary (recall that this is not required for the existence of $\mathfrak{d}(\cdot, m)$).

\begin{lemma}[Fourier-type integrals for the bias of $\mathfrak{d}(\cdot,m)$] \label{IntegraldarstBiasDecDens}
Suppose absolute continuity of $\chi$, continuity of $f_{\chi\ast\bareps}$ and validity of Assumption \ref{Ass2024121902}. Then,
\begin{align} \label{2024121504}
\mathfrak{d}(\xi,m) - \frac{ f_\chi(\xi+) + f_\chi(\xi-) }{2} &= \liml_{\delta \downarrow 0} \mathfrak{R}_\delta(m, \xi) \hspace{1cm} (\xi \in \R, ~ m \geq 0),
\end{align}
where we denote
\begin{align} \label{2024121505}
\mathfrak{R}_\delta(m, \xi) := -\frac{1}{2\pi} \intl_{-\infty}^\infty e^{-i\xi t} \Phi_\iota(\delta t) \Phi_\chi(t) \mathcal{P}_\bareps(t, m) dt.
\end{align}
\end{lemma}

\begin{proof}
In the present situation, by definition (\ref{DecDens1}) and Lemma \ref{BasicPropDecfct}, we can write $\mathfrak{d}(\xi, m) - f_\chi(\xi) = - \int_{-\infty}^\infty f_\chi(\xi-x) F_{\mu_\bareps}^{\ast (m+1)}(dx)$, for all $\xi \in \R$. The asserted formula thus follows from Theorem \ref{TheoInvDens}.
\end{proof}

The performance of the deconvolution density is illustrated in Figure \ref{density_deconvolution_fourier.fig}.

\begin{figure}
\centering
\resizebox*{\textwidth}{!}{\includegraphics{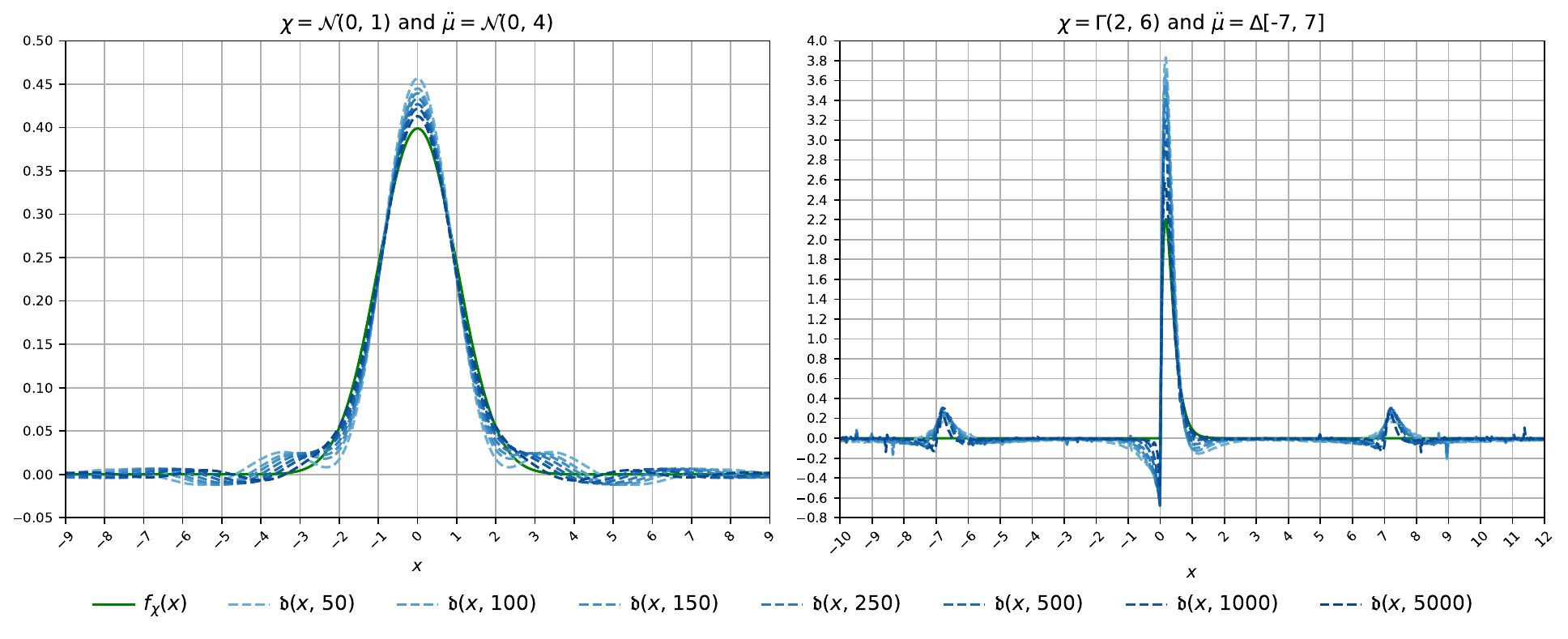}}
\vspace{-0.5cm}
\caption{The deconvolution density for various absolutely continuous distributions. In the right plot, the narrow character of the true density $f_\chi(x)$ causes serious inaccuracies in the approximation, as well as the discontinuity at $x=0$.} \label{density_deconvolution_fourier.fig}
\end{figure}

\section{Convergence of the deconvolution function} \label{ChConvDecFct}

We already mentioned, since the deconvolution function is associated with a signed measure, that the convergence of the corresponding Fourier-Stieltjes transform is insufficient to conclude the convergence of $\mathfrak{D}(\cdot,m)$ to $F_\chi$. Therefore, in the present paragraph we study the convergence behaviour of the deconvolution function and its possibly existing density, for which we deploy the Fourier integrals from \S\ref{SecFourierIntDecFct}. A first inspection of these integrals suggests that convergence properties essentially depend on the involved c.fs.. Roughly speaking, one can distinguish between the integral being absolutely and uniformly convergent with respect to $T> 0$ or existing merely as the limit of a sequence of integrals. In the first case, it is easy to verify a strong kind of convergence.

\begin{theorem}[uniform convergence of $\mathfrak{D}(\cdot, m)$] \label{SatzglmKvgzDm}
Assume validity of Assumption \ref{Ass2026070901} and the existence of $T > 0$, with
\begin{align} \label{IntCondPhiX}
\intl_T^\infty \frac{\abs{\Phi_\chi(t)}}{t} dt<\infty.
\end{align}
Then,
\begin{align*}
\liml_{m \rightarrow \infty} \norm{ \mathfrak{D}(\cdot, m) - F_\chi }_\infty \in [0,\infty).
\end{align*}
The limit equals zero, if $\mathcal{N}_\mu \setminus \mathcal{N}_\chi$ has zero Lebesgue measure.
\end{theorem}

Clearly, a necessary condition for the applicability of the last theorem is the continuity of $F_\chi$. The theorem reveals the effect of a c.f. $\Phi_\mu$ with a compact support. It is then not possible to recover the blurred d.f. $F_\chi$, unless also $\Phi_\chi$ has a compact support. More detailed statements on this case can be found below. Finally, it turns out through Theorem \ref{SatzglmKvgzDm} that the convergence to zero neither depends on the existence of any moments nor on the support of $F_\chi$, which is quite remarkable. For instance, suppose that $F_\bareps$ has moments up to order $K_\bareps$ and $F_\chi$ has moments up to order $K_\chi$, for $K_\bareps < K_\chi$. It then follows from the properties of convolution that $\mathfrak{D}(\cdot,m)$ has moments up to order $K_\bareps$. But if $\mathfrak{D}(\cdot,m)\rightarrow F_\chi$, as $m \rightarrow \infty$, in the limit, we return to a function with moments up to order $K_\chi$. Similarly, if $F_\chi$ has a finite support. In these circumstances, regardless of $F_\bareps$, the support of $\mathfrak{D}(\cdot,m)$ is either infinite or increases as $m$ increases, whereas the limit function, if convergence to zero holds, has a finite support.

\begin{proof}[Proof of Theorem \ref{SatzglmKvgzDm}]
The existence of (\ref{IntCondPhiX}) implies the decay of $\Phi_\chi(t)$, as $t \rightarrow \pm\infty$, so that $\chi$ must be associated with a continuous distribution, according to Corollary 2 to Theorem 3.2.3 in \cite{Lukacs1970}. But then $F_\chi\ast F_\bareps^{\ast j}$ is also continuous, for any $j \in \N_0$, so that $\mathfrak{D}(\xi, m) = \mathfrak{D}(\xi-, m)$ and $F_\chi(\xi) = F_\chi(\xi-)$, for all $\xi\in\R$. Hence, Lemma \ref{BiasDekfkt} immediately yields
\begin{align} \nonumber
\norm{ \mathfrak{D}(\cdot, m) - F_\chi }_\infty &\leq \frac{1}{\pi} \intl_{[0,\infty)\setminus \mathcal{N}_\mu} \frac{ \mathcal{P}_\bareps(t,m) }{t} \abs{\Phi_\chi(t)} dt + \frac{1}{\pi} \intl_{[0,\infty)\cap \mathcal{N}_\mu} \frac{ \abs{\Phi_\chi(t)} }{t} dt.
\end{align}
Observing that $0\leq \mathcal{P}_\bareps(t,m) \leq \mathcal{P}_\bareps(t,0)$, under Assumption \ref{Ass2026070901} and (\ref{IntCondPhiX}), it is clear that the first integrand is bounded by an integrable function, which does not depend on $m$. In addition, $\lim_{m \rightarrow \infty} \mathcal{P}_\bareps(t,m) = 0$, for all $t\in[0,\infty)\setminus \mathcal{N}_\mu$, so that Lebesgue's dominated convergence theorem implies the decay of this integral. Regarding the se\-cond integral, we note that $0 \notin \mathcal{N}_\mu$, since $\Phi_\bareps(0)=1$, and that $\Phi_\chi(t)$ is continuous along the real axis. Hence, the denominator in the integral is bounded away from zero and the integral is finite, due to (\ref{IntCondPhiX}). More precisely, it equals either zero or a finite positive constant, whose magnitude depends on $\mathcal{N}_\mu$ and on $\Phi_\chi$.
\end{proof}

The proof of an analogous result for the deconvolution density is straightforward.

\begin{corollary}[uniform convergence of $\mathfrak{d}(\cdot, m)$]
In the situation of Lemma \ref{IntegraldarstBiasDecDens}, if $\Phi_\chi \in L^1(\R)$,
\begin{align*}
\liml_{m \rightarrow \infty} \norm{ \mathfrak{d}(\cdot, m) - f_\chi }_\infty \in [0, \infty).
\end{align*}
The limit equals zero if $\mathcal{N}_\mu \setminus \mathcal{N}_\chi$ has zero Lebesgue measure.
\end{corollary}

\begin{proof}
We remark that $\Phi_\chi \in L^1(\R)$ implies uniform continuity of $f_\chi(\xi)$ with respect to $\xi \in \R$, due to the inversion formula for Fourier transforms. Further arguments are the same as in the proof of Theorem \ref{SatzglmKvgzDm}.
\end{proof}

D.fs. that satisfy the conditions of the previous theorems play an outstanding role, however, they are necessarily of continuous type. Conversely, already a single $\chi$-atom suffices to violate absolute convergence of the Fourier-type integrals for $\mathfrak{D}(\cdot,m)$. Indeed, the challenge in a general $\chi$ is that the associated c.f. $\Phi_\chi$ need not contribute to absolute convergence. Thus, Lebesgue's dominated convergence theorem can not easily be applied. In the sequel, we discuss various non-trivial convergence statements. Possibly the simplest of these corresponds to the case of a compactly supported $\Phi_\bareps$, which is then a symmetric interval, since $\Phi_\bareps$ is even.

\begin{theorem}[convergence of $\mathfrak{D}(\cdot,m)$ for compactly supported $\Phi_\bareps$] \label{TheoConvDecFctCompSup}
Suppose validity of Assumption \ref{Ass2026070901} and the existence of $T_\bareps > 0$ with $\Phi_\bareps(t) = 0$, for all $\abs{t} > T_\bareps$. Define
\begin{align} \label{PointwConv02c}
\Theta(t) := \frac{ \Phi_\chi(t) }{i2\pi t} \ONE_{ [-T_\bareps, T_\bareps] \cap \mathcal{N}_\mu }(t).
\end{align}
Then, for every $\xi \in \R$,
\begin{align} \label{PointwConv07a}
\liml_{m\rightarrow\infty} \frac{\mathfrak{D}(\xi, m)+\mathfrak{D}(\xi-, m)}{2} &= \frac{ F_\chi(\xi) + F_\chi(\xi-) }{2} + \mathcal{F}\curbr{\Theta}(-\xi) + \mathcal{F}\curbr{ \Psi_{T_\bareps, \infty} }(-\xi),
\end{align}
where the last term refers to (\ref{2024122001}). In particular, the second and the third summand are uniformly bounded with respect to $\xi$ in any compact subset of $\R$.
\end{theorem}

The last theorem shows, for $\Phi_\bareps$ with a compact support, that the bias between the deconvolution function and the target d.f. always converges to a finite limit. The limit is particularly small if $T_\bareps$ is large, if $[-T_\bareps, T_\bareps] \cap \mathcal{N}_\mu$ is of zero Lebesgue measure or if $\Phi_\chi(t)$ decays, as $t \rightarrow \pm\infty$.

\begin{proof}
In the present situation, for fixed $T>T_\bareps$ and $\xi \in \R$, through Lemma \ref{IntegraldarstDm}, we get
\begin{align} \label{PointwConv02}
\frac{\mathfrak{D}(\xi, m)+\mathfrak{D}(\xi-, m)}{2} &- \frac{ F_\chi(\xi) + F_\chi(\xi-) }{2} = \operatorname{I}_m + \mathcal{F}\curbr{\Theta}(-\xi) + \liml_{T \rightarrow \infty} \mathcal{F}\curbr{ \Psi_{T_\bareps, T} }(-\xi),
\end{align}
where we denote
\begin{align} \label{PointwConv02b}
\operatorname{I}_m &:= \frac{1}{2\pi i} \intl_{ [-T_\bareps,T_\bareps]\setminus \mathcal{N}_\mu  } \frac{\mathcal{P}_\bareps(t,m)}{t} e^{-i\xi t} \Phi_\chi(t) dt.
\end{align}
The integral $\operatorname{I}_m$ is the only component with $\mathcal{P}_\bareps(t,m)<1$. Since $0\notin \mathcal{N}_\mu$, its range of integration includes some neighborhood of $t=0$. Moreover, the integrand is uniformly bounded with respect to $m$ by $t^{-1}\mathcal{P}_\bareps(t,0)$, which is integrable, by Assumption \ref{Ass2026070901} and by finiteness of $T_\bareps$. Hence, Lebesgue's dominated convergence theorem implies the decay of $\operatorname{I}_m$, as $m \rightarrow \infty$. Furthermore, the finiteness of $T_\bareps$ combined with the continuity of $\Theta(t)$ imply the existence of the Fourier transform $\mathcal{F}\curbr{\Theta}(-\xi)$, for every $\xi \in \R$. It is even absolutely and uniformly convergent in any compact subset of the real axis. Lastly, the required statements on the Fourier transform $\mathcal{F}\curbr{\Psi_{T_\bareps, T}}(-\xi)$ immediately follow from Lemma \ref{Lem2024121901}, thereby finishing the proof.
\end{proof}

With regard to the deconvolution density, the analogue to the previous theorem requires additional assumptions.

\begin{assumption}[differentiability of smoothing kernel] \label{Ass2025011301}
The c.f. $\Phi_\iota$ from Assumption \ref{Ass2024121902} is continuously differentiable, for $\abs{t} > 0$, and there exists $a > 2$, with $\Phi_\iota'(t) = \LandauO(\abs{t}^{-a})$, as $t \rightarrow \pm \infty$. 
\end{assumption}

Our next assumption is fulfilled, e.g., if $\Phi_\chi$ can be decomposed into an oscillatory and a monotonic factor.

\begin{assumption}[factorization of $\Phi_\chi$ (case $\mathfrak{d}(\cdot, m)$)] \label{Ass2024121901}
There exist $T_0 \geq 0$, $\kappa \in \mathcal{M}(\C, \mathcal{B}(\R))$, with $|\kappa|(\R) = |\kappa|(D_\kappa)$, and a continuous function of bounded variation $\varphi_\chi : [T_0, \infty] \rightarrow \C$, with $\lim_{t\rightarrow\infty}\varphi_\chi(t)=0$, such that $\Phi_\chi(t) = \Phi_\kappa(t)\varphi_\chi(t)$, whenever $\abs{t} \geq T_0$.
\end{assumption}

By construction, $\kappa$ obviously must be a discrete complex measure. The factorization summarizes two cases. On the one hand, $\Phi_\kappa$ can be the c.f. of a discrete and $\varphi_\chi$ the c.f. of a continuous distribution. Yet, neither of them actually needs to be a c.f.. As an example, consider $\Phi_\chi(t) = t^{-1} \sin(t)$. In this event, Assumption \ref{Ass2024121901} applies, with $\kappa = (2i)^{-1} (\delta_{\curbr{1}} - \delta_{\curbr{-1}})$ and $\varphi_\chi(t) = t^{-1}$.

\begin{theorem}[convergence of $\mathfrak{d}(\cdot,m)$ for compactly supported $\Phi_\bareps$] \label{TheoConvDecDensCompSup}
In the situation of Lemma \ref{IntegraldarstBiasDecDens}, suppose the existence of $T_\bareps > 0$, with $\Phi_\bareps(t) = 0$, for all $\abs{t} > T_\bareps$. Then, under Assumptions \ref{Ass2025011301} and \ref{Ass2024121901}, denoting
\begin{align} \label{2024121501}
\theta(t) :=&\, \frac{ 1 }{2\pi} \Phi_\chi(t) \ONE_{ [-T_\bareps, T_\bareps] \cap \mathcal{N}_\mu }(t),
\end{align}
for every $\xi \in \R$ with $\Delta(\curbr{\xi}, D_\kappa) > 0$, we have
\begin{align} \label{2024121503}
\liml_{m\rightarrow\infty} \mathfrak{d}(\xi,m) = \frac{ f_\chi(\xi+) + f_\chi(\xi-) }{2} - \mathcal{F}\curbr{\theta}(-\xi) - \mathcal{F}\curbr{ \psi_{T_\bareps, 0} }(-\xi),
\end{align}
where the Fourier transform $\mathcal{F}\curbr{ \psi_{T_\bareps, 0} }$ was evaluated in (\ref{2024122002}).
\end{theorem}

\begin{proof}
We proceed from Lemma \ref{IntegraldarstBiasDecDens}. It is clear that
\begin{align} \label{2024121506}
\mathfrak{R}_\delta(m, \xi) = \operatorname{R}_{\delta, m} - \mathcal{F}\curbr{\Phi_\iota(\delta\cdot)\theta}(-\xi) - \mathcal{F}\curbr{ \psi_{T_\bareps, \delta} }(-\xi),
\end{align}
the right hand side referring to
\begin{align*}
\operatorname{R}_{\delta, m} := -\frac{1}{2\pi} \intl_{ [-T_\bareps, T_\bareps] \setminus \mathcal{N}_\mu } e^{-i\xi t} \Phi_\iota(\delta t) \Phi_\chi(t)\mathcal{P}_\bareps(t,m) dt.
\end{align*}
By dominated convergence, $\operatorname{R}_{\delta, m}$ approaches a finite limit, as $\delta \downarrow 0$, and eventually vanishes, as $m \rightarrow \infty$. The remaining addends in (\ref{2024121506}) only depend on $\delta$. As $\delta \downarrow 0$, we have $\mathcal{F}\curbr{\Phi_\iota(\delta\cdot)\theta}(-\xi) \rightarrow \mathcal{F}\curbr{\theta}(-\xi)$, by continuity of $\theta$. Finally, concerning $\mathcal{F}\curbr{ \psi_{T_\bareps, \delta} }(-\xi)$, a reference to Lemma \ref{Lem2024121902} completes the proof.
\end{proof}

\subsection{Arguments of weak convergence} \label{SectWeakConvmpwr}

Particularly if $\Phi_\bareps(t)$ is non-increasing with respect to $t \geq 0$ and vanishing at infinity, it becomes obvious, for fixed $m \geq 0$, that the graph of the $m$-power $\mathcal{P}_\bareps(\cdot, m)$ on the positive real axis resembles that of a non-negative d.f.. In such a case, it is reasonable to expect that the convergence of $\mathfrak{D}(\cdot,m)$ can be justified by weak convergence. To this end, suppose that the limit $\Phi_\bareps(\infty)$ exists and hence $\mathcal{P}_\bareps(\infty,m) \in [0,1]$. Additionally, assume that $\mathcal{P}_\bareps([0, \infty], m) < \infty$, for each $m \geq 0$. Then, since $\Phi_\bareps(0)=1$, by definition of a c.f., we can write
\begin{align} \label{mpwrasapproxid1}
\mathcal{P}_\bareps(t,m) = \intl_{[0,t]} \mathcal{P}_\bareps(dv,m) \hspace{1cm} (t \in [0,\infty], ~ m \geq 0).
\end{align}
This function, if $\Phi_\bareps(t)$ is once continuously differentiable on $[0,\infty]$, even possesses a density with respect to the Lebesgue measure, viz
\begin{align} \label{mpwrasapproxid2}
\frac{d}{dv} \mathcal{P}_\bareps(v,m) = -(m+1)\Phi_\bareps'(v) \mathcal{P}_\bareps(v,m-1).
\end{align}
In any case, by continuity of $\Phi_\bareps(t)$, also $\mathcal{P}_\bareps(t, m)$ is continuous with respect to $t \in [0,\infty]$ and the integral signs $\int_{[0,t]}$, $\int_{(0,t)}$ and $\int_0^t$ have the same meaning in (\ref{mpwrasapproxid1}). The transition $m \rightarrow \infty$, however, forces us to employ the notion of a compact interval, especially if $\Phi_\bareps$ vanishes at one of the endpoints. More precisely, we observe that
\begin{align} \label{mpwrasapproxid3}
\mathcal{P}_\bareps(t,\infty) := \liml_{m \rightarrow \infty} \mathcal{P}_\bareps(t,m) = 
\begin{cases}
1, & \mbox{if } t \in [0,\infty] \cap \mathcal{N}_\mu, \\
0, & \mbox{if } t \in [0,\infty] \setminus \mathcal{N}_\mu.
\end{cases}
\end{align}
Evidently, this function establishes a signed measure of discrete type, and a point $t \in [0,\infty]$ has mass one if $\Phi_\bareps(t)$ vanishes there and mass zero otherwise. Unlike $\mathcal{P}_\bareps(t,m)$, the limit measure therefore exhibits discontinuities. It can be expressed in terms of indicator functions. For this, and especially for the actual applicability of the above limit statement to the deconvolution function, we impose the assumption below.

\begin{assumption} \label{mpwrasapproxid3b}
$\Phi_\bareps(\infty)$ exists and $\mathcal{P}_\bareps(\cdot,m)$ is of bounded variation on $[0,\infty]$, uniformly with respect to $m\geq0$, formally
\begin{align} \label{mpwrasapproxid4}
\supl_{m \geq 0} \abs{ \mathcal{P}_\bareps }([0, \infty],m) < \infty. 
\end{align}
In addition, there exists $K \in \N$ and a sequence $(\tau_k)_{k\in I} \subseteq \mathcal{N}_\mu\cap [0,\infty]$, for consecutive integers $I := \rrb{1, 2, \hdots, 2K}$, such that
\begin{align*}
\begin{cases}
\tau_k \leq \tau_{k+1} \mbox{ and } [\tau_k, \tau_{k+1}] \subseteq \mathcal{N}_\mu, & \mbox{for odd } k \in I, \\
\tau_k < \tau_{k+1}, & \mbox{for even } k \in I,
\end{cases}
\end{align*}
as well as
\begin{align*}
\mathcal{N}_\mu \cap [0,\infty] = \bigcup\limits_{ k=1}^K [\tau_{2k-1},\tau_{2k}].
\end{align*}
\end{assumption}

It is easy to find examples that satisfy (\ref{mpwrasapproxid4}), e.g., if $\Phi_\bareps$ decays monotonically towards infinity. In any case, $\tau_1>0$, since $0\notin \mathcal{N}_\mu$, and a segment $[\tau_k,\tau_{k+1}]$, for odd $k \in I$, is either an isolated point or a continuous segment of the positive real axis, where $\Phi_\bareps$ vanishes. We can now make the following statement.

\begin{lemma}[weak convergence of the $m$-power] \label{mpwrasapproxid5c}
Under Assumption \ref{mpwrasapproxid3b}, for any continuous function $u : [0, \infty] \rightarrow \R$, we have
\begin{align*}
\liml_{m \rightarrow \infty} \intl_{[0,\infty]} u(v) \mathcal{P}_\bareps(dv,m) = \suml_{k=1}^{K-1} \rrb{ u(\tau_{2k-1}) - u(\tau_{2k}) } + u(\tau_{2K-1}) - \IndNr{\tau_{2K} \, < \, \infty} u(\tau_{2K}).
\end{align*}
\end{lemma}

The proof essentially relies on the Helly-Bray theorem. Yet, since $\mathcal{P}_\bareps(\cdot, m)$ is not necessarily monotonic, rather than the probabilistic, we require the general version (see, e.g., Theorem 16.4 in Ch. 1 in \cite{widder1946}).

\begin{proof}
First, a comparison with (\ref{mpwrasapproxid3}), for $t \in [0,\infty]$, shows that
\begin{align} \label{mpwrasapproxid3e}
\mathcal{P}_\bareps(t,\infty) &= \suml_{k=1}^{K-1} \cb{ \IndNr{t \, \geq \, \tau_{2k-1}} - \IndNr{t \, > \, \tau_{2k}} } + \cb{ \IndNr{t \, \geq \, \tau_{2K-1}} - \IndNr{\infty \, \geq \, t \, > \, \tau_{2K}} }.
\end{align}
We remark that the last indicator vanishes, if $\tau_{2K}=\infty$. If also $\tau_{2K-1}=\infty$, the second last indicator equals one if and only if $t=\infty$. Now, the validity of (\ref{mpwrasapproxid4}) admits a reference to the Helly-Bray theorem. Accordingly, $\mathcal{P}_\bareps(t,\infty)$ is of bounded variation on $[0,\infty]$, and, for all with respect to $v \in [0,\infty]$ continuous functions $u(v)$, we have
\begin{align*}
\operatorname{L} := \liml_{m \rightarrow \infty} \intl_{[0,\infty]} u(v) \mathcal{P}_\bareps(dv,m) &= \intl_{[0,\infty]} u(v) \mathcal{P}_\bareps(dv,\infty).
\end{align*}
To evaluate the integral on the right hand side, we require an appropriate representation for the sum (\ref{mpwrasapproxid3e}). On the one hand, for $0\leq \tau \leq \infty$, it is clear that $\IndNr{t \, \geq \, \tau}$ is the d.f. associated with the Dirac measure with mass at $\tau$, i.e., $\IndNr{t \, \geq \, \tau} = \delta_{\rrb{\tau}}([0,t])$. On the other hand, for $0\leq \tau<\infty$, $\IndNr{t \, > \, \tau}$ corresponds to the limit of a sequence of such measures. In particular,
\begin{align*}
\IndNr{t \, > \, \tau} = \liml_{\substack{\eta \downarrow 0\\ \eta > 0}} \delta_{\rrb{\tau+\eta}}([0,t]).
\end{align*}
Therefore, by (\ref{mpwrasapproxid3e}),
\begin{align*}
\operatorname{L} &= \suml_{k=1}^{K-1} \intl_{[0,\infty]} u(v) \cb{ \delta_{\rrb{\tau_{2k-1}}}(dv) - \lim_{\substack{\eta \downarrow 0\\ \eta > 0}} \delta_{\rrb{\tau_{2k}+\eta}}(dv) } \\ \nonumber
&\hspace{2cm} + \intl_{[0,\infty]} u(v) \delta_{\rrb{\tau_{2K-1}}}(dv) - \IndNr{\tau_{2K} \, < \, \infty} \intl_{[0,\infty]} u(v) \lim_{\substack{\eta \downarrow 0\\ \eta > 0}} \delta_{\rrb{\tau_{2K}+\eta}}(dv).
\end{align*}
Since $\delta_{\rrb{\tau+\eta}}([0,\infty])=1$, for any $\eta>0$, we identify $\delta_{\rrb{\tau+\eta}}([0,t])$ as a sequence of functions of bounded variation on $[0,\infty]$, uniformly with respect to $\eta\geq 0$. Thus, again as a consequence of the Helly-Bray theorem and by continuity of $u(v)$, we obtain
\begin{align*}
\intl_{[0,\infty]} u(v) \lim_{\substack{\eta \downarrow 0\\ \eta > 0}} \delta_{\rrb{\tau+\eta}}(dv) &= \lim_{\substack{\eta \downarrow 0\\ \eta > 0}} \intl_{[0,\infty]} u(v) \delta_{\rrb{\tau+\eta}}(dv) \\ \nonumber
&= \lim_{\substack{\eta \downarrow 0\\ \eta > 0}} u(\tau+\eta) \\ \nonumber
&= u(\tau).
\end{align*}
Altogether, the proof is completed.
\end{proof}

We next apply the previous result to establish pointwise convergence of deconvolution function and density.

\begin{theorem}[pointwise convergence of $\mathfrak{D}(\cdot, m)$ I] \label{TheoPWConv1}
Under Assumption \ref{mpwrasapproxid3b}, for any $\xi \in \R$, we have
\begin{align} \label{EqPWConv1}
\liml_{m \rightarrow \infty} \frac{ \mathfrak{D}(\xi, m) + \mathfrak{D}(\xi-, m) }{2} = \frac{ F_\chi(\xi) + F_\chi(\xi-) }{2} + \suml_{k=1}^K \mathcal{F}\curbr{ \Psi_{ \tau_{2k-1}, \tau_{2k} } }(-\xi),
\end{align}
where the sum involves the Fourier transform (\ref{ConvPropDecFct3b}). It equals zero, if $\mathcal{N}_\mu$ has Lebesgue measure zero.
\end{theorem}

Note that the conditions of the theorem are especially satisfied if $\mathcal{N}_\mu \cap \R = \emptyset$, and if there exists $t_0>0$ for which $\Phi_\bareps(t)$ exhibits monotonicity on $\abs{t} > t_0$.

\begin{proof}[Proof of Theorem \ref{TheoPWConv1}]
For brevity, we write $\operatorname{D}(\xi,m) := \frac{1}{2} \curbr{ \mathfrak{D}(\xi, m) + \mathfrak{D}(\xi-, m) }$ and $G(t,\xi) := \Im e^{i\xi t} \Phi_\chi(-t)$. Then, for $\xi \in \R$, elementary manipulations of the integral from Lemma \ref{IntegraldarstDm}, upon accounting for (\ref{FTDm1}), yield
\begin{align} \label{ConvPropDecFct1}
\operatorname{D}(\xi,m) = \frac{1}{2} + \liml_{ \substack{ T_1 \downarrow 0 \\ T_2 \uparrow \infty } } \intl_{[T_1,T_2]} \frac{ G(t, \xi) }{ \pi t } \rrb{ 1 - \mathcal{P}_\bareps(t,m) } dt.
\end{align}
Observe that $\int_{T_1}^t (\pi s)^{-1} G(s, \xi) ds = - \mathcal{F}\curbr{\psi_{T_1, t}}(-\xi)$, for $t \geq T_1$. Thus, with $0 < T_1 < T_2 < \infty$, partial integration leads to
\begin{align} \nonumber
\intl_{[T_1,T_2]} \rrb{ 1 - \mathcal{P}_\bareps(t,m) } \frac{ G(t, \xi) }{ \pi t } dt &= - \rrb{ 1 - \mathcal{P}_\bareps(T_2,m) } \mathcal{F}\curbr{ \Psi_{ T_1, T_2 } }(-\xi) \\ \label{ConvPropDecFct3}
&\hspace{1.5cm} - \intl_{[T_1,T_2]} \mathcal{F}\curbr{ \Psi_{ T_1, t } }(-\xi) \mathcal{P}_\bareps(dt,m).
\end{align}
In Lemma \ref{Lem2024121901}, for any fixed $T_1 \geq 0$, the Fourier transform $\mathcal{F}\curbr{\Psi_{T_1, t}}(-\xi)$ was verified as a uniformly continuous function of $t \in [0, \infty]$. Furthermore, under the theorem's conditions, $\mathcal{P}_\bareps(\infty,m)$ exists. Hence, if we combine (\ref{ConvPropDecFct1}) with (\ref{ConvPropDecFct3}), for $\xi\in \R$, we get
\begin{align*}
\operatorname{D}(\xi,m) = \frac{1}{2} - \rrb{ 1 - \mathcal{P}_\bareps(\infty,m) } \mathcal{F}\curbr{ \Psi_{ 0, \infty } }(-\xi) - \intl_{[0,\infty]} \mathcal{F}\curbr{ \Psi_{ 0, t } }(-\xi) \mathcal{P}_\bareps(dt,m).
\end{align*}
As $m \rightarrow \infty$, the curved bracket in the second summand either vanishes or tends to unity, depending on whether or not $\infty \in \mathcal{N}_\mu$. Also, it is easy to see that
\begin{align*}
\mathcal{F}\curbr{ \Psi_{ 0, T } }(-\xi) - \mathcal{F}\curbr{ \Psi_{ 0, S } }(-\xi) = \mathcal{F}\curbr{ \Psi_{ S, T } }(-\xi) \hspace{1cm} (S \leq T).
\end{align*}
According to Lemma \ref{mpwrasapproxid5c}, we thus arrive at
\begin{align} \nonumber
\liml_{m \rightarrow \infty} \operatorname{D}(\xi,m) &= \frac{1}{2} - \mathcal{F}\curbr{ \Psi_{ 0, \infty } }(-\xi)\IndNr{\infty \, \notin \, \mathcal{N}_\mu} + \suml_{k=1}^{K-1} \mathcal{F}\curbr{ \Psi_{ \tau_{2k-1}, \tau_{2k} } }(-\xi) \\ \label{ConvPropDecFct4}
&\hspace{2cm} - \mathcal{F}\curbr{ \Psi_{ 0, \tau_{2K-1} } }(-\xi) + \IndNr{\tau_{2K} \, < \, \infty} \mathcal{F}\curbr{ \Psi_{ 0, \tau_{2K} } }(-\xi).
\end{align}
Note that $\mathcal{F}\curbr{ \Psi_{ 0, \infty } }(-\xi) = \frac{1}{2} \curbr{ 1 - F_\chi(\xi) - F_\chi(\xi-) }$, by (\ref{2024122001}). Finally, since $\infty \notin \mathcal{N}_\mu$ implies that $\tau_{2K}<\infty$, the right hand side then matches (\ref{EqPWConv1}). Conversely, if $\infty \in \mathcal{N}_\mu$, then necessarily $\tau_{2K}=\infty$ and the second and the last summand in (\ref{ConvPropDecFct4}) both vanish. But we always have $\mathcal{F}\curbr{ \Psi_{ 0, \tau_{2K-1} } }(-\xi) = \mathcal{F}\curbr{ \Psi_{ 0, \infty } }(-\xi) - \mathcal{F}\curbr{ \Psi_{ \tau_{2K-1}, \infty } }(-\xi)$, which again validates (\ref{EqPWConv1}).
\end{proof}

We proceed with the analogue statement for the deconvolution density.

\begin{theorem}[pointwise convergence of $\mathfrak{d}(\cdot, m)$ I] \label{TheoPWConvdm1}
In the situation of Lemma \ref{IntegraldarstBiasDecDens}, under Assumptions \ref{Ass2025011301}, \ref{Ass2024121901} and \ref{mpwrasapproxid3b}, for any $\xi \in \R$ with $\Delta( \curbr{\xi}, D_\kappa ) > 0$, we have
\begin{align} \label{EqPWConvdm1}
\liml_{m \rightarrow \infty} \mathfrak{d}(\xi,m) &= \frac{ f_\chi(\xi+) + f_\chi(\xi-) }{2} - \suml_{k=1}^K \rrb{ \mathcal{F}\curbr{ \psi_{ \tau_{2k-1}, 0 } }(-\xi) - \mathcal{F}\curbr{ \psi_{ \tau_{2k}, 0 } }(-\xi) },
\end{align}
where the sum involves the Fourier transform (\ref{2024122002}), with $\mathcal{F}\curbr{ \psi_{ \infty, 0 } }(-\xi) = 0$. The sum equals zero, if $\mathcal{N}_\mu$ has Lebesgue measure zero.
\end{theorem}

\begin{proof}
Starting from Lemma \ref{LemFourIntDecDens}, analogous to the first part of the proof of Theorem \ref{TheoPWConv1}, one can show that
\begin{align*}
\mathfrak{d}(\xi,m) &= \liml_{\delta \downarrow 0} \rrb{ 1 - \mathcal{P}_\bareps(\infty,m) } \mathcal{F}\curbr{ \psi_{ 0, \delta } }(-\xi) \\
&\hspace{3cm} + \liml_{\delta \downarrow 0} \intl_{[0,\infty]} \curbr{ \mathcal{F}\curbr{ \psi_{ 0, \delta } }(-\xi) - \mathcal{F}\curbr{ \psi_{ t, \delta } }(-\xi) } \mathcal{P}_\bareps(dt,m).
\end{align*}
Due to Lemma \ref{Lem2024121902}, the limit as $\delta \downarrow 0$ is permissible under the sign of integration, and $\mathcal{F}\curbr{ \psi_{ t, 0 } }(-\xi)$ is uniformly continuous with respect to $t \in [0, \infty]$. Hence, Lemma \ref{mpwrasapproxid5c} applies. Particularly note, in view of Theorem \ref{TheoInvDens}, for $\xi \in \R$, that $\mathcal{F}\curbr{ \psi_{ 0, 0 } }(-\xi) = \frac{1}{2} \curbr{ f_\chi(\xi+) + f_\chi(\xi-) }$.
\end{proof}

Clearly, Assumption \ref{mpwrasapproxid3b} does not cover the quite common case in which $\Phi_\bareps$ has infinitely many zeros. The next example shows that the assumption is then not only violated due to the actual number of zeros, but that even the required uniformity of the bounded variation may not be taken for granted.

\begin{example} \label{ExampSqSin}
If $\bareps$ is associated with a triangular distribution on $[-2,2]$, then $\Phi_\bareps(t) = t^{-2} \curbr{ \sin(t) }^2$. Hence, $t\mapsto \mathcal{P}_\bareps(t,m)$ is absolutely continuous, for each $m \geq 0$. We first compute the variation on $[\pi k, \pi(k+1)]$, for $k \in \N$. The derivative $\Phi_\bareps'(t)$ on $(\pi k, \pi(k+1))$ has only one zero, denoted by $t_k$, to the right and left of which it is increasing and decreasing, respectively. Thus, by (\ref{mpwrasapproxid2}),
\begin{align*}
\frac{ \abs{ \mathcal{P}_\bareps }([\pi k, \pi(k+1)],m) }{m+1} = \intl_{\pi k}^{t_k} \Phi_\bareps'(t) \mathcal{P}_\bareps(t,m-1) dt - \intl_{t_k}^{\pi (k+1)} \Phi_\bareps'(t) \mathcal{P}_\bareps(t,m-1) dt.
\end{align*}
Straightforwardly, we evaluate these integrals by reference to the fundamental theorem of calculus, from what we get
\begin{align*}
\abs{ \mathcal{P}_\bareps }([\pi k, \pi(k+1)],m) &= \cb{ - \mathcal{P}_\bareps(t,m) }_{\pi k}^{t_k} + \cb{ \mathcal{P}_\bareps(t,m) }_{t_k}^{\pi(k+1)} = 2 (1- \mathcal{P}_\bareps(t_k,m)).
\end{align*}
In sum, it shows that the variation on $[0,\infty]$ equals
\begin{align*}
\abs{\mathcal{P}_\bareps}([0, \infty],m) = - (m+1) \intl_0^\pi \Phi_\bareps'(t) \mathcal{P}_\bareps(t,m-1) dt + 2 \suml_{k=1}^\infty (1- \mathcal{P}_\bareps(t_k,m)).
\end{align*}
The series on the right hand side converges, for every finite $m \geq 0$, since $1 - \mathcal{P}_\bareps(t,m) = \LandauO(t^{-2})$, as $t \rightarrow \infty$, by (\ref{2024122301}). But all summands are non-negative and $\sup_{m \geq 0} (1- \mathcal{P}_\bareps(t_k,m)) = 1$, for each $k \in \N$. Consequently, $\mathcal{P}_\bareps([0,\infty], m)$ is not uniformly bounded, thereby violating the entire Assumption \ref{mpwrasapproxid3b}.
\end{example}

The previous example suggests a general problem with $m$-powers composed of c.fs. that vanish on an infinite countable set of points. Apparently, in such cases, a reference to the Helly-Bray theorem is infeasible. We tackle this issue in the next paragraph.

\subsection{Test for pointwise convergence by means of alternating sums}

Often non-absolutely convergent Fourier-type integrals are still finite due to oscillatory contributions. Our results so far on pointwise convergence of the deconvolution function and density also build on the presence of oscillatory terms, however, mostly those that arise from $\Phi_\chi$ (compare, e.g., Lemma \ref{Lem2024121901} and \ref{Lem2024121902}). In contrast, oscillatory behaviour of $\Phi_\bareps$ and hence of $\mathcal{P}_\bareps(\cdot, m)$ has not yet been exploited. Furthermore, in the presence of this kind of behaviour, we can not even expect the applicability of our earlier results, e.g., from \S\ref{SectWeakConvmpwr}. These indeed will be useless if $\Phi_\bareps$ is periodic. For that reason, in this paragraph, we separately study such scenarios, which we first formalize.

\begin{assumption} \label{PointwConvAltSums3}
There exist constants $\rho>0$ and $j_\bareps \in \N_0$, for which $\mathcal{P}_\bareps(t+j\rho,m)$ is non-decreasing or non-increasing with respect to integer $j \geq j_\bareps$, for each fixed $0 \leq t \leq \rho$.
\end{assumption}

Clearly, the parameter $\rho>0$ corresponds to some kind of period. C.fs. that fulfill the above assumption are products of monotonic and periodic functions, e.g., $\Phi_\bareps(t) = t^{-2} \curbr{ \sin(t) }^2$, as well as mixtures of the form
\begin{align} \label{edwcftamospamf1}
\Phi_\bareps = a \Phi_{ \bareps_d } + (1-a) \Phi_{ \bareps_c },
\end{align}
for $0 < a \leq 1$ and c.fs. $0 \leq \Phi_{ \bareps_d }, \Phi_{ \bareps_c } \leq 1$, of which $\Phi_{ \bareps_d }$ is periodic, while $\Phi_{ \bareps_c }$ is monotonic. Now, frequently used techniques to extract oscillatory ingredients from an integral include partial integration and a sophisticated partitioning of the range of integration. Here, we additionally combine these methods with Abelian summation by parts. While integration by parts may rearrange a non-absolutely convergent to an absolutely convergent integral, Abelian summation by parts provides the analogue for series. It essentially enters the proof of the following auxiliary statement. In the sequel, for $\kappa \in \mathcal{M}(\C, \mathcal{B}(\R))$, we write
\begin{align} \label{2026071001}
\mathcal{H}_{\rho, \xi}^\kappa := \rrb{ (x-\xi)\rho : x \in D_\kappa }.
\end{align}

\begin{lemma} \label{LemPeriodmPwr}
Suppose validity of Assumption \ref{PointwConvAltSums3} and that $\mathcal{N}_\mu$ is of Lebesgue measure zero. For $\kappa \in \mathcal{M}(\C, \mathcal{B}(\R))$, with $|\kappa|(\R) = |\kappa|(D_\kappa)$, define
\begin{align} \label{PointwConvAltSums2d}
q_{t, m}(s) := \mathcal{P}_\bareps(s,m) \Phi_\kappa(s) \IndNr{ T_0 \leq s \leq t } \hspace{1cm} (t \geq T_0 \geq j_\bareps\rho).
\end{align}
Then, $\mathcal{F}\curbr{ q_{t, m} }(-\xi)$ is a continuous function of $t \geq T_0$, with $\mathcal{F}\curbr{ q_{T_0, m} }(-\xi) = 0$, for any $\xi \in \R$. In particular, $\lim_{m \rightarrow \infty} \mathcal{F}\curbr{q_{t, m}}(-\xi) = 0$, for any $t \geq T_0$, and there exists $K> 0$ with $| \mathcal{F}\curbr{q_{t, m}}(-\xi) | \leq K$, uniformly with respect to $t \geq T_0$ and $m \geq 0$, for each $\xi \in \R$ with $\Delta(\mathcal{H}_{\rho, \xi}^\kappa , 2\pi\Z ) > 0$.
\end{lemma}

\begin{proof}
The first statement follows from the continuity of $q_{t, m}(s)$ with respect to $T_0 \leq s \leq t$. In order to verify the asserted limit, defining $K_1:= \sqrt{2} (\abs{ \Re\kappa }(\R) + \abs{ \Im\kappa }(\R))$, we observe that $\norm{ \Phi_\kappa }_\infty \leq K_1$ and thus
\begin{align} \label{PointwConvAltSums12}
\abs{ \mathcal{F}\curbr{ q_{t, m} }(-\xi) } \leq K_1 \intl_{T_0}^t \mathcal{P}_\bareps(s,m) ds.
\end{align}
But $\mathcal{P}_\bareps(s,m)\leq 1$, uniformly with respect to $s \in \R$, $m\geq0$, and $\mathcal{P}_\bareps(s,m) \rightarrow 0$, as $m \rightarrow \infty$, for Lebesgue almost any $s \in \R$. Hence, according to dominated convergence, the upper bound (\ref{PointwConvAltSums12}) vanishes, as $m \rightarrow \infty$, for any fixed $t \geq T_0$. To confirm the uniform boundedness of $\mathcal{F}\curbr{ q_{t, m} }(-\xi)$, we define $J_t := \max\rrb{ j\in\N_0 : j\rho\leq t }$. Furthermore, assume that $T_0 := J_0\rho$, for an arbitrary integer $J_0 \geq j_\bareps$. Then, $J_t \geq J_0$, whenever $t \geq T_0$. Now, upon dividing the range of integration of $\mathcal{F}\curbr{ q_{t, m} }(-\xi)$ into a countable number of segments, according to the periodic component of the $m$-power, accompanied by two substitutions, we get
\begin{align} \label{PointwConvAltSums6}
\mathcal{F}\curbr{ q_{t, m} }(-\xi) &= \intl_0^\rho e^{-i\xi (s+J_0\rho)} \sigma_{J_t}(s,m,\xi) ds + \intl_{J_t\rho}^t \mathcal{P}_\bareps(s,m) e^{-i\xi s} \Phi_\kappa(s) ds,
\end{align}
where we defined
\begin{align} \label{PointwConvAltSums5c}
\sigma_{J_t}(s,m,\xi) :=&\, \suml_{j=0}^{J_t-J_0-1} \mathcal{P}_\bareps(s+(J_0+j)\rho,m) e^{-i\xi j \rho} \Phi_\kappa(s+(J_0+j)\rho).
\end{align}
In this sum, to separate the $m$-power from the trigonometric factors, we introduce
\begin{align} \label{PointwConvAltSums7}
\operatorname{C}(n, s,\xi) := \suml_{j=0}^n e^{-i\xi j \rho} \Phi_\kappa(s+(J_0+j)\rho) \hspace{1cm} (n \in \N_0).
\end{align}
Then, defining
\begin{align*}
u(n,s,m) := \mathcal{P}_\bareps(s+(J_0+n+1)\rho,m) - \mathcal{P}_\bareps(t+(J_0+n)\rho,m),
\end{align*}
by means of the Abelian sum formula (see $\S\S$182-183, pp. 322-323 in \cite{Knopp_1976}), the representation (\ref{PointwConvAltSums5c}) becomes
\begin{align} \label{PointwConvAltSums8}
\sigma_{J_t}(s,m,\xi) &= \mathcal{P}_\bareps(s+J_t\rho,m) \operatorname{C}(J_t-J_0-1,s,\xi) - \suml_{n=0}^{J_t-J_0-1} \operatorname{C}(n,s,\xi) u(n,s,m).
\end{align}
In terms of the integral definition of $\Phi_\kappa$, we can write
\begin{align*}
\operatorname{C}(n,s,\xi) = \intl_{-\infty}^\infty e^{ix(s+J_0\rho)} \suml_{j=0}^n e^{i(x-\xi) j \rho} \kappa(dx).
\end{align*}
Thereof, with the aid of the formula for geometric sums, we deduce that
\begin{align*}
\operatorname{C}(n,s,\xi) = \intl_{-\infty}^\infty \frac{\sin\rrb{\frac{(x-\xi)\rho(n+1)}{2}}}{\sin\rrb{\frac{(x-\xi)\rho}{2}}} e^{ix(s+\frac{2J_0+n}{2}\rho) - i\xi\frac{\rho n}{2}} \kappa(dx).
\end{align*}
Generally, the above ratio of sine functions is $\LandauO(n)$, for any $x \in D_\kappa$ with $(x-\xi)\rho \in 2\pi\Z$. Yet, due to the assumption $\Delta(\mathcal{H}_{\rho, \xi}^\kappa, 2\pi\Z) > 0$, the denominator is bounded away from zero and $\abs{\operatorname{C}(n,s,\xi)} \leq K_2$, for some constant $K_2>0$, uniformly with respect to $n\in\N_0$ and $0\leq s \leq \rho$. As a consequence, in view of the assumed monotonicity of $\mathcal{P}_\bareps(\cdot, m)$ and due to its uniform boundedness, the sum (\ref{PointwConvAltSums8}) satisfies the bound
\begin{align*}
\abs{ \sigma_{J_t}(s,m,\xi) } &\leq K_2 \mathcal{P}_\bareps(s+J_t\rho,m) + K_2 \abs{ \mathcal{P}_\bareps(s+J_t\rho,m) - \mathcal{P}_\bareps(s+J_0\rho,m) } \\
&\leq 3K_2.
\end{align*}
It shows the finiteness and especially the boundedness of the sequence of partial sums $\sigma_{J_t}(s,m,\xi)$, uniformly with respect to $0\leq s \leq \rho$, $J_t \geq J_0$ and $m \geq 0$. Moreover, concerning the second integral in (\ref{PointwConvAltSums6}), uniformly with respect to $t \geq T_0$ and $m\geq 0$, we have
\begin{align*}
\abs{\intl_{J_t\rho}^t \mathcal{P}_\bareps(s,m) e^{-i\xi s} \Phi_\kappa(s) ds } \leq (t-J_t\rho) K_1 \leq \rho K_1.
\end{align*}
The second inequality holds, since $0\leq t-J_t\rho <\rho$, by definition. To summarize these findings, by (\ref{PointwConvAltSums6}), uniformly with respect to $m \geq 0$ and $t \geq T_0$, it was just verified that $\abs{ \mathcal{F}\curbr{ q_{t, m} }(-\xi) } \leq 3\rho K_2 + \rho K_1$, which completes the proof.
\end{proof}

To facilitate an application of the previous lemma to the bias of deconvolution function or density, we need an appropriate factorization of the integrand, according to oscillatory and vanishing components. With regard to the deconvolution function, similar to Assumption \ref{Ass2024121901}, we therefore impose the following requirement.

\begin{assumption}[factorization of $\Phi_\chi$ (case $\mathfrak{D}(\cdot, m)$)] \label{PointwConvAltSums2}
There exist $T_0\geq 0$, $\kappa \in \mathcal{M}(\C, \mathcal{B}(\R))$, with $|\kappa|(\R) = |\kappa|(D_\kappa)$, and a function $\varphi_\chi : [T_0, \infty]\rightarrow \C$, such that $\Phi_\chi(t)=\Phi_\kappa(t)\varphi_\chi(t)$, for all $t \geq T_0$, and $t^{-1} \varphi_\chi(t)$ is continuous, of bounded variation on $[T_0,\infty]$ and vanishing, as $t \rightarrow \infty$.
\end{assumption}

Again, like in Assumption \ref{Ass2024121901}, $\kappa$ is a discrete complex measure. However, in contrast, we observe that $\varphi_\chi(t)\equiv1$ is possible, so that $\Phi_\chi$ especially can be purely oscillatory or even constant, i.e., associated with a degenerate distribution. It is now straightforward to establish the next theorem.

\begin{theorem}[pointwise convergence of $\mathfrak{D}(\cdot, m)$ II] \label{PointwConvDecFctII}
Under Assumptions \ref{Ass2026070901}, \ref{PointwConvAltSums3} and \ref{PointwConvAltSums2}, if $\mathcal{N}_\mu$ is of Lebesgue measure zero, it holds that
\begin{align*}
\liml_{m\rightarrow\infty} \frac{ \mathfrak{D}(\xi,m) + \mathfrak{D}(\xi-,m) }{2} = \frac{ F_\chi(\xi) + F_\chi(\xi-) }{2},
\end{align*}
for any $\xi \in \R$, with $\Delta(\mathcal{H}_{\rho, \xi}^\kappa , 2\pi\Z ) > 0$.
\end{theorem}

A comparison with Theorem \ref{TheoPWConv1} reveals, contrary to monotonic $m$-powers, that the pointwise convergence in the presence of periodicity only happens subject to additional restrictions on the local parameter, which avoid conflicts of the fluctuations, i.e., cancellations and thereby possible divergence.

\begin{proof}[Proof of Theorem \ref{PointwConvDecFctII}]
We start with a transformation of (\ref{PointwConv01b}) to an integral along the positive real axis only, that is, for fixed $T>T_0>0$, $m\geq 0$ and $\xi \in \R$,
\begin{align} \label{PointwConvAltSums1}
\mathfrak{I}_T(m,\xi) &= \frac{1}{\pi} \Im \rrb{ \operatorname{L}_{0, T_0}(m,\xi) + \operatorname{L}_{T_0, T}(m,\xi) },
\end{align}
where we defined
\begin{align} \label{PointwConvAltSums1Y}
\operatorname{L}_{T_1,T_2}(m,\xi) := \intl_{T_1}^{T_2} \frac{\mathcal{P}_\bareps(t,m)}{t} e^{-i\xi t} \Phi_\chi(t) dt \hspace{1cm} (T_2 > T_1 \geq 0).
\end{align}
For brevity, we write $a(t) := t^{-1} \varphi_\chi(t)$. Moreover, denoting
\begin{align*}
J_0 :=&\, \min\rrb{ j \in \N_0 : j \geq j_\bareps \mbox{ and } a(t) \mbox{ is of bounded variation on } [j\rho, \infty] },
\end{align*}
we agree that $T_0 := J_0\rho$. Now, it is obvious that the Fourier transform $\mathcal{F}\curbr{ q_{t, m} }(-\xi)$ of the function (\ref{PointwConvAltSums2d}) corresponds to the antiderivative of $s \mapsto e^{-i\xi s}\mathcal{P}_\bareps(s,m) \Phi_\kappa(s)$ on $[T_0, t]$. Therefore, for $T \geq T_0$, through integration by parts, we get
\begin{align} \label{PointwConvAltSums2c}
\operatorname{L}_{T_0,T}(m,\xi) &= a(T) \mathcal{F}\curbr{ q_{T, m} }(-\xi) - \intl_{T_0}^T \mathcal{F}\curbr{ q_{t, m} }(-\xi) a(dt).
\end{align}
In view of Lemma \ref{LemPeriodmPwr}, for fixed $m\geq 0$ and $\xi \in \R$, the Fourier transform $\mathcal{F}\curbr{ q_{t, m} }(-\xi)$ is uniformly bounded with respect to $t\geq 0$. Besides, $a(T) \rightarrow 0$ as $T \rightarrow \infty$, by assumption. Hence, in (\ref{PointwConvAltSums2c}), the first summand vanishes and, since $a(t)$ is of bounded variation on $[T_0,\infty]$, the integral with respect to $a(t)$ converges absolutely and uniformly with respect to $T\geq T_0$. Upon combining (\ref{BiasDekfkt}), (\ref{PointwConvAltSums1}) and (\ref{PointwConvAltSums2c}), for fixed $\xi \in \R$ and $m\geq 0$, concisely writing $\operatorname{D}(\xi, m) := \frac{1}{2} \curbr{ \mathfrak{D}(\xi,m)+\mathfrak{D}(\xi-,m)) }$, we thus arrive at
\begin{align*}
\operatorname{D}(\xi, m) - \frac{ F_\chi(\xi) + F_\chi(\xi-) }{2} = \frac{1}{\pi} \Im \operatorname{L}_{0, T_0}(m, \xi) - \frac{1}{\pi} \Im \intl_{T_0}^\infty \mathcal{F}\curbr{ q_{t, m} }(-\xi) a(dt).
\end{align*}
Under the current assumptions, as $m \rightarrow \infty$, the decay of the first summand is trivial. With regard to the second term, due to Lemma \ref{LemPeriodmPwr} and dominated convergence, the limit can be carried out under the integral sign, and the limit value equals zero. The proof is thus finished.
\end{proof}

Our final theorem is the analogue to Theorem \ref{PointwConvDecFctII} for the deconvolution density.

\begin{theorem}[pointwise convergence of $\mathfrak{d}(\cdot, m)$ II]
In the situation of Lemma \ref{IntegraldarstBiasDecDens}, under Assumptions \ref{Ass2025011301}, \ref{Ass2024121901} and \ref{PointwConvAltSums3}, if $\mathcal{N}_\mu$ is of Lebesgue measure zero,
\begin{align*}
\liml_{m\rightarrow\infty}\mathfrak{d}(\xi,m) = \frac{ f_\chi(\xi+) + f_\chi(\xi-) }{2},
\end{align*}
for any $\xi \in \R$, with $\Delta(\mathcal{H}_{\rho, \xi}^\kappa , 2\pi\Z ) > 0$.
\end{theorem}

\begin{proof}
Consider fixed $\delta > 0$, $m \geq 0$ and $\xi \in \R$. Define
\begin{align*}
\operatorname{I}_\delta(m,\xi) := \intl_{T_0}^\infty e^{-i\xi t} \Phi_\iota(\delta t) \varphi_\chi(t) \Phi_\kappa(t) \mathcal{P}_\bareps(t, m) dt.
\end{align*}
Then, elementary manipulations of (\ref{2024121505}), due to Assumption \ref{Ass2024121901}, show that
\begin{align*}
\mathfrak{R}_\delta(m, \xi) = -\frac{1}{\pi} \Re \rrb{ \intl_0^{T_0} e^{-i\xi t} \Phi_\iota(\delta t) \Phi_\chi(t) \mathcal{P}_\bareps(t, m) dt + \operatorname{I}_\delta(m,\xi) }.
\end{align*}
On the right hand side, the first integral approaches a finite limit as $\delta \downarrow 0$ that eventually vanishes as $m \rightarrow \infty$. Concerning the second integral, with the aid of $\Phi_\iota(\delta t) = -\delta \int_t^\infty \Phi_\iota'(\delta s) ds$, we receive
\begin{align*}
\operatorname{I}_\delta(m,\xi) = - \intl_{\delta T_0}^\infty \Phi_\iota'(s) \intl_{T_0}^\frac{s}{\delta} \varphi_\chi(t) \Phi_\kappa(t) \mathcal{P}_\bareps(t, m) dt ds,
\end{align*}
the interchange in the order of integration being permissible, due to the asymptotic behaviour of $\Phi_\iota'$. The interior integral, as a function of $\delta^{-1} s$, can be treated similar to the integral from the proof of Theorem \ref{PointwConvDecFctII}. In this fashion, it eventually follows that the limit as $\delta \downarrow 0$ can be carried out under the sign of integration, with
\begin{align*}
\liml_{\delta \downarrow 0} \operatorname{I}_\delta(m,\xi) = - \intl_{T_0}^\infty \mathcal{F}\curbr{ q_{t, m} }(-\xi) \varphi_\chi(dt).
\end{align*}
Finally, according to Lemma \ref{LemPeriodmPwr} and since $\varphi_\chi$ is of bounded variation on $[T_0, \infty]$, the limit as $m \rightarrow \infty$ also can be evaluated under the sign of integration, provided $\xi \in \R$ with $\Delta(\mathcal{H}_{\rho, \xi}^\kappa, 2 \pi \Z) > 0$. Thus, the proof is completed.
\end{proof}

\section{Conclusion} \label{SubSummary}

To summarize our findings, in this text, we essentially introduced the deconvolution function $\mathfrak{D}(\cdot, m)$ and studied its convergence properties, as $m \rightarrow \infty$. Thereby, we were especially able to justify a novel formula for the recovery of an arbitrary d.f. $F_\chi$ from the convolution transform $F_{\chi\ast\mu}$, for a broad range of probability measures $\chi$ and $\mu$. The deconvolution function is defined as the convolution of the transformed d.f. $F_{\chi \ast \bareps}$ with the so-called deconvolution kernel $\mathfrak{K}_\bareps(\cdot, m)$, a Neumann sum of convolution powers with truncation index $m$, whose Fourier-Stieltjes transform corresponds to the partial sum of a geometric series. In view of the obtained convergence behaviour, $\mathfrak{D}(\xi, m)$ yields an approximation for $F_\chi(\xi)$, for each fixed $(\xi,m) \in \R\times\N_0$. Thus, if $\chi$ is associated with a degenerate distribution, the function $\mathfrak{D}(\cdot, m)$ is an approximate identity. Furthermore, $\mathfrak{K}_\bareps(\cdot, m)$ corresponds to an approximation for the inverse of $F_\bareps$ with respect to convolution. Similar results were established in the density domain, however, at the cost of slightly restrictive conditions. For most of our study, we deployed the representation of $\mathfrak{D}(\cdot, m)$ as a Fourier-type integral, as it is more convenient and numerically more stable compared to the Neumann sum. This representation also furnishes a good starting point for an asymptotic analysis, e.g., to quantify the rate of convergence, in order to assess the magnitude of the bias. In the Fourier domain, the bias $\mathfrak{D}(\cdot, m) - F_\chi$ (see \S\ref{SecFourierIntDecFct}) represents a generalized Laplace transform with phase function $(m+1)\log(1-\Phi_\bareps)$. Accordingly, the main contribution for large $m$ to the total value of the bias comes from a neighborhood of the minima of $\Phi_\bareps$, since the kernel $\mathcal{P}_\bareps(\cdot,m)$ is at these points greater than elsewhere. More precisely, the graph of the kernel at these so-called critical points becomes sharper as $m$ grows, whereas the remaining area becomes relatively negligible. Integrals with this property are well-known as Laplace-type. They are usually evaluated by Laplace's method (see \cite{Olver1974, paris_2011}), aiming for an approximation of the integrand near the critical points, through the coefficients in the respective asymptotic expansions. The exact asymptotic behaviour of the bias $\mathfrak{D}(\cdot, m)-F_\chi$ thus sensitively depends on the local structure of the c.fs. $\Phi_\chi$ and $\Phi_\bareps$. A straightforward estimate shows that the rate of convergence is of exponential order, if $\Phi_\bareps$ does not vanish in the closure of the support of $\Phi_\chi$. Yet, this is rarely the case. Instead, the derivation of asymptotic statements can be quite challenging, because standard techniques for Laplace-type integrals may not be applicable, as the frequently required existence of local power series expansions is too restrictive. For example, if $\Phi_\bareps$ is exponentially small near a critial point compared with $\Phi_\chi$, power series approximations for both functions are obviously unavailable. To overcome such difficulties, we propose the use of Mellin-Barnes integrals (see \cite{KamPar2001}). With the aid of these, in \cite{Kaiser04102025}, we established va\-rious results, through which one can specify the rate for the convergence of $\norm{ \mathfrak{D}(\cdot, m) - F_\chi}_\infty$, as $m \rightarrow\infty$, if at least one of the probability measures $\chi$ and $\bareps$ is associated with a stable distribution. Results that moreover aim for a wider applicability were developed in \cite{kaiser2026generalizinglaplacesmethodmeans}.

\backmatter

\section*{Declarations}

\bmhead{Funding} This work did not receive a funding.

\begin{appendices}

\section{The sine integral} \label{AppSiInt}

Of major importance in Fourier analysis is the sine integral
\begin{align} \label{Integralsinus01}
\Si(\xi) := \intl_0^\xi \frac{\sin(t)}{t} dt \hspace{1cm} (\xi \in \R).
\end{align}
One can easily verify that $\Si(0)=0$ and $\Si(a\xi) = \sgn(a) \Si(\abs{a}\xi)$, for every $\xi>0$ and $a\in\R\setminus\rrb{0}$, as well as that (see, e.g., Ch. 2, $\S$3.3, p. 41 in \cite{Olver1974})
\begin{align} \label{2025011101}
\lim_{\xi \rightarrow \infty} \Si(a\xi) = \sgn(a) \frac{\pi}{2} \hspace{1cm} (a \in\R).
\end{align}
Hence, $\Si : \overline \R \rightarrow \overline \R$ is uniformly continuous. Finally, elementary computations show that
\begin{align} \label{Integralsinus02c}
0 \leq \abs{\Si(\xi)} \leq \Si(\pi) \hspace{1cm} (\xi\in\R).
\end{align}

\section{Finiteness of special Fourier transforms} \label{AppFinSpecFT}

We here examine the convergence behaviour of two Fourier transforms. The first of them plays a key role in the recovery of a d.f. from its c.f. and also in our investigations of the convergence of the deconvolution function. The properties of this Fourier transform are essentially derived from the sine integral, compare Appendix \ref{AppSiInt}.

\begin{lemma} \label{Lem2024121901}
For an arbitrary c.f. $\Phi_\chi$, define
\begin{align} \label{PointwConv03}
\Psi_{S, T}(t) := \frac{ \Phi_\chi(t) }{i2\pi t} \IndNr{ S \leq \abs{t} \leq T } \hspace{1cm} (T > S > 0).
\end{align}
Then,
\begin{align} \label{ConvPropDecFct3b}
\mathcal{F}\curbr{\Psi_{S, T}}(-\xi) &= - \frac{1}{\pi} \intl_{ -\infty }^\infty \rrb{\Si((\xi-x)T) - \Si((\xi-x)S)} F_\chi(dx) \hspace{1cm} (\xi \in \R)
\end{align}
and $\sup_{S, T > 0}| \mathcal{F}\curbr{\Psi_{S, T}}(-\xi) | \leq \pi^{-1} \Si(\pi)$. In particular, for fixed $T > S > 0$ and $\xi \in \R$, each of the following limits exists and can be computed from (\ref{ConvPropDecFct3b}) under the integral sign:
\begin{align*}
\mathcal{F}\curbr{\Psi_{0, T}}(-\xi) &:= \lim_{S \downarrow 0} \mathcal{F}\curbr{\Psi_{S, T}}(-\xi) \\
\mathcal{F}\curbr{\Psi_{S, \infty}}(-\xi) &:= \lim_{T \rightarrow \infty} \mathcal{F}\curbr{\Psi_{S, T}}(-\xi) \\
\mathcal{F}\curbr{\Psi_{0, \infty}}(-\xi) &:= \lim_{ S \downarrow 0 } \mathcal{F}\curbr{\Psi_{S, \infty}}(-\xi)
\end{align*}
Specifically $\lim_{ T \rightarrow \infty } \mathcal{F}\curbr{\Psi_{0, T}}(-\xi) = \lim_{ S \downarrow 0 } \mathcal{F}\curbr{\Psi_{S, \infty}}(-\xi)$ and
\begin{align} \label{2024122001}
\mathcal{F}\curbr{\Psi_{S, \infty}}(-\xi) = \frac{1 - F_\chi(\xi) - F_\chi(\xi-)}{2} + \frac{1}{\pi} \intl_{ -\infty }^\infty \Si((\xi-x)S) F_\chi(dx).
\end{align}
\end{lemma}

\begin{proof}
Upon invoking the integral definition of $\Phi_\chi$, for $T > S > 0$ and $\xi \in \R$, we get
\begin{align*}
\mathcal{F}\curbr{\Psi_{S, T}}(-\xi) &= - \intl_{ -\infty }^\infty \intl_S^T \frac{e^{is(\xi-x)} - e^{-is(\xi-x)}}{i2\pi s} ds F_\chi(dx),
\end{align*}
where the interchange in the order of integration is permitted due to the boundedness of the integrand. After substitution, writing the result in terms of the sine integral (\ref{Integralsinus01}), we arrive at (\ref{ConvPropDecFct3b}). Thus, the oddness of the sine integral and inequality (\ref{Integralsinus02c}) imply the indicated bound. Furthermore, in view of the convergence properties of the sine integral, Lebesgue's dominated convergence theorem eventually admits the evaluation of $\mathcal{F}\curbr{\Psi_{S, T}}(-\xi)$ in the limits $S \downarrow0$ and $T \rightarrow \infty$ under the integral sign, in arbitrary order. Specifically due to (\ref{2025011101}), we have
\begin{align*}
\liml_{ T \rightarrow \infty\ } \frac{1}{\pi} \intl_{ -\infty }^\infty \Si((\xi-x)T) F_\chi(dx) &= \frac{1}{2} \intl_{ -\infty }^\infty \sgn(\xi-x) F_\chi(dx) \\
&= \frac{1}{2} \rb{F_\chi(\xi-) -  (1-F_\chi(\xi))},
\end{align*}
which finishes the proof.
\end{proof}

The next Fourier transform is required to establish convergence of the deconvolution density.

\begin{lemma} \label{Lem2024121902}
For an arbitrary c.f. $\Phi_\chi$, under Assumptions \ref{Ass2024121902}, \ref{Ass2025011301} and \ref{Ass2024121901}, define
\begin{align} \label{2024121502}
\psi_{\tau, \delta}(t) := \frac{ 1 }{2\pi} \Phi_\iota(\delta t) \Phi_\chi(t) \IndNr{ \abs{t} \geq \tau } \hspace{1cm} (\tau \geq 0, ~ \delta > 0).
\end{align}
Then, for any $\tau \geq 0$ and $\xi \in \R$ with $\Delta( \curbr{\xi}, D_\kappa) > 0$, there exists $K > 0$, such that $\sup_{\delta > 0} | \mathcal{F}\curbr{\psi_{\tau, \delta}}(-\xi) | \leq K$, and the limit $\mathcal{F}\curbr{\psi_{\tau, 0}}(-\xi) := \lim_{\delta \downarrow 0} \mathcal{F}\curbr{\psi_{\tau, \delta}}(-\xi)$ exists, with
\begin{align} \label{2024122002}
\mathcal{F}\curbr{\psi_{\tau, 0}}(-\xi) = \IndNr{\tau < T_0} \mathcal{F}\curbr{\psi_\tau}(-\xi) + \Re \theta_\tau(-\xi),
\end{align}
where $T_0$ is as in Assumption \ref{Ass2024121901}, $\psi_\tau(t) := (2\pi)^{-1} \IndNr{\tau \leq \abs{t} \leq T_0} \Phi_\chi(t)$ and
\begin{align} \label{2024122003}
\theta_\tau(-\xi) := \intl_{-\infty}^\infty\intl_{\max\curbr{T_0, \tau}}^\infty \frac{ e^{i(x-\xi) \max\curbr{T_0, \tau} } - e^{i(x-\xi) u } }{ i\pi(x-\xi) } \varphi_\chi(du) \kappa(dx).
\end{align}
\end{lemma}

\begin{proof}
First of all, because $\overline{ \psi_{\tau, \delta}(t) } = \psi_{\tau, \delta}(-t)$ and $2\Re z = z + \overline z$, for all $z \in \C$, we observe that
\begin{align*}
\mathcal{F}\curbr{\psi_{\tau, \delta}}(-\xi) = \IndNr{\tau < T_0} \frac{1}{2\pi}\intl_{ \curbr{ \tau \leq \abs{t} \leq T_0 } } e^{-i\xi t} \Phi_\iota(\delta t) \Phi_\chi(t) dt + \Re \theta_{\delta, \tau}(-\xi),
\end{align*}
for every $\tau \geq 0$, $\delta > 0$ and $\xi \in \R$, in terms of
\begin{align*}
\theta_{\delta, \tau}(-\xi) := \frac{1}{\pi}  \intl_{\max\curbr{T_0, \tau}}^\infty e^{-i\xi t} \Phi_\iota(\delta t) \Phi_\chi(t) dt.
\end{align*}
It is obvious that the first summand in the above decomposition is uniformly bounded and tends to $\mathcal{F}\curbr{\psi_\tau}(-\xi)$, as $\delta \downarrow 0$. With regard to the second addend, we denote
\begin{align*}
g_{\tau, s}(t) := \Phi_\chi(t) \IndNr{ \max\curbr{T_0, \tau} \leq t \leq s } \hspace{1cm} (s \geq \max\curbr{T_0, \tau} ).
\end{align*}
Moreover, since $\Phi_\iota(t)$ is continuously differentiable and vanishing at infinity, we can write $\Phi_\iota(\delta t) = -\delta \int_t^\infty \Phi_\iota'(\delta s) ds$, whenever $t > 0$. Thereof, by additional substitution, we get
\begin{align} \label{2024121903}
\theta_{\delta, \tau}(-\xi) = - \frac{1}{\pi} \intl_{ \delta \max\curbr{T_0, \tau} }^\infty \Phi_\iota'(s) \mathcal{F}\curbr{ g_{\tau, \frac{s}{\delta}} }(-\xi) ds,
\end{align}
where the interchange in the order of integration is permissible, due to the asymptotic behaviour of $\Phi_\iota'$ and because $| \mathcal{F}\curbr{ g_{\tau, s} }(-\xi)| \leq s$. Notice that the function $g_{\tau, s}$ differs from (\ref{PointwConv03}) by a reciprocal $t$-power. Consequently, uniform boundedness with respect to $s > 0$ of the associated Fourier transform is not natural. However, by Assumption \ref{Ass2024121901}, $\Phi_\chi(t)$ admits a factorization into an oscillatory and a decreasing component, for all $t \geq T_0$. Thus,
\begin{align*}
\mathcal{F}\curbr{ g_{\tau, s} }(-\xi) &= \intl_{-\infty}^\infty \intl_{ \max\curbr{T_0, \tau} }^s e^{i(x-\xi)t} \varphi_\chi(t) dt \kappa(dx) \hspace{1cm} (\xi \in \R, ~ s \geq \max\curbr{T_0, \tau}).
\end{align*}
Because $\varphi_\chi(t)$ is of bounded variation on $[T_0, \infty]$ and vanishing, as $t \rightarrow \infty$, we have $\varphi_\chi(t) = -\int_{t}^\infty \varphi_\chi(du)$ and thereby
\begin{align*}
\mathcal{F}\curbr{ g_{\tau, s} }(-\xi) &= \intl_{-\infty}^\infty\intl_{ \max\curbr{T_0, \tau} }^\infty \frac{ e^{i(x-\xi) \max\curbr{T_0, \tau} } - e^{i(x-\xi)\min\curbr{s, u}} }{ i(x-\xi) } \varphi_\chi(du) \kappa(dx).
\end{align*}
The integrand of the interior integral is uniformly bounded with respect to $s \geq \max\curbr{T_0, \tau}$ and $x \in D_\kappa$. A simple use of the triangle inequality thus shows that $| \mathcal{F}\curbr{ g_{\tau, s} }(-\xi) | \leq K_1$, for a constant $K_1 > 0$, uniformly with respect to $s \geq \max\curbr{T_0, \tau}$. By application to (\ref{2024121903}), we arrive at $| \theta_{\delta, \tau}(-\xi) | \leq K_1 \int_0^\infty |\Phi'(s)| ds < \infty$. Furthermore, as $\delta \downarrow 0$, the limit of $\mathcal{F}\curbr{ g_{\tau, \frac{s}{\delta}} }(-\xi)$ is admissible under the sign of integration. Finally, also in (\ref{2024121903}), the limit as $\delta \downarrow 0$ can be performed under the integral sign, with
\begin{align*}
\liml_{\delta \downarrow 0} \theta_{\delta, \tau}(-\xi) = - \theta_\tau(-\xi) \intl_0^\infty \Phi_\iota'(s) ds.
\end{align*}
But $\int_0^\infty \Phi_\iota'(s) ds = -\Phi_\iota(0) = -1$, because $\Phi_\iota$ is a c.f.. Altogether, the limit of $\mathcal{F}\curbr{\psi_{\tau, \delta}}(-\xi)$, as $\delta \downarrow 0$, therefore has the asserted form (\ref{2024122002}).
\end{proof}

\section{Inversion of Fourier transforms} \label{AppInvCF}

Numerous inversion formulae are available for Fourier transforms, since each fa\-mi\-ly of functions requires different criteria \cite{koerner_1988, Lukacs1970, Pinsky2002, titchmarsh1937}. With regard to the focus of this text, we here confine to d.fs. and densities of complex measures. Corresponding inversion formulae immediately can be obtained from those for probabiliy measures, upon separating real and imaginary part and appealing to Jordan's decomposition (Theorem 9.30 in \cite{axler2019measure}). Technically speaking, our first two formulae are consequences of Lemma \ref{Lem2024121901}. Particularly the first of them is a barely known result due to \cite{GilPelaez1951}.

\begin{theorem}[unilateral inversion formula for d.f.] \label{SatzInvHalbstr}
For any $\chi \in \mathcal{M}(\C, \mathcal{B}(\R))$, it holds that
\begin{align*}
\frac{F_\chi(\xi) + F_\chi(\xi-)}{2} = \frac{1}{2} + \liml_{ \substack{ T_1 \downarrow 0 \\ T_2 \uparrow \infty } } \frac{1}{2\pi} \intl_{T_1}^{T_2} \frac{e^{i\xi t}\Phi_\chi(-t) - e^{-i\xi t}\Phi_\chi(t)}{it} dt \hspace{1cm} (\xi \in \R),
\end{align*}
where the order of the limits is arbitrary.
\end{theorem}

We proceed with the standard formula for recovering a d.f. from the associated Fourier-Stieltjes transform, which may not be missing in any textbook on Fourier methods in probability theory. It generalizes Theorem 3.2.1 in \cite{Lukacs1970} and corresponds to Theorem 2.3.11 in \cite{Pinsky2002}, if $F_\chi$ is absolutely continuous. The formula basically can be obtained from Theorem \ref{SatzInvHalbstr}, by considering the difference between two points. In this way, one avoids the appearance of a second limit.

\begin{theorem}[bilateral inversion formula for d.f.] \label{SatzInvZuw}
For any $\chi \in \mathcal{M}(\C, \mathcal{B}(\R))$, with $\phi_{a,b}$ as in (\ref{cFUab}) and finite $a < b$, we have
\begin{align*}
\frac{ F_\chi(b) + F_\chi(b-) }{2} - \frac{ F_\chi(a) + F_\chi(a-) }{2} = \liml_{T \rightarrow \infty} \frac{\sgn(b-a)}{2\pi} \intl_{-T}^T \phi_{a,b}(-t) \Phi_\chi(t) dt.
\end{align*}
\end{theorem}

In the previous theorem, the finiteness of $a$ and $b$ is crucial, since the limit of the integrand, e.g., as $a\rightarrow -\infty$, is always unspecified. Lastly, we also mention a slight modification of Corollary 3 to Theorem 3.3.2 in \cite{Lukacs1970} as a means to recover functions from the space $L^1(\R)$, i.e., densities of absolutely continuous complex measures.

\begin{theorem}[inversion formula for $L^1(\R)$] \label{TheoInvDens}
Under Assumption \ref{Ass2024121902}, for any $q \in L^1(\R)$ and $\xi \in \R$, with finite $q(\xi+)$ and $q(\xi-)$, we have
\begin{align*}
\frac{ q(\xi+) + q(\xi-) }{2} = \liml_{\delta \downarrow 0} \frac{1}{2\pi} \intl_{-\infty}^\infty e^{-i\xi t} \Phi_\iota(\delta t) \mathcal{F}\curbr{q}(t) dt.
\end{align*}
\end{theorem}

\end{appendices}

\bibliography{References.bib}

\end{document}